\documentclass[11pt]{amsart}

\usepackage{a4wide}
\usepackage{amssymb}
\usepackage{mathrsfs}
\usepackage{enumerate}
\usepackage{esint}
\usepackage{titletoc}
\usepackage[colorlinks=true, urlcolor=blue, linkcolor=blue, citecolor=black]{hyperref}
\usepackage[nameinlink]{cleveref}

\theoremstyle{plain}
\newtheorem{theorem}{Theorem}[section]
\newtheorem{lemma}[theorem]{Lemma}
\newtheorem{corollary}[theorem]{Corollary}
\newtheorem{proposition}[theorem]{Proposition}

\theoremstyle{definition}
\newtheorem{definition}[theorem]{Definition}
\newtheorem{remark}[theorem]{Remark}

\numberwithin{equation}{section}

\DeclareMathOperator{\supp}{supp}

\DeclareMathOperator*{\essinf}{ess\,inf}
\DeclareMathOperator*{\wsup}{w-sup}
\DeclareMathOperator*{\winf}{w-inf}

\makeatletter
\@namedef{subjclassname@2020}{\textup{2020} Mathematics Subject Classification}
\makeatother

\title{The Wiener criterion for nonlocal Dirichlet problems}

\author{Minhyun Kim}
\address{Fakult\"at f\"ur Mathematik, Universit\"at Bielefeld, Bielefeld 33615, Germany}
\email{minhyun.kim@uni-bielefeld.de}

\author{Ki-Ahm Lee}
\address{Department of Mathematical Sciences \& Research Institute of Mathematics, Seoul National University, 08826 Seoul, Republic of Korea}
\email{kiahm@snu.ac.kr}

\author{Se-Chan Lee}
\address{Department of Mathematical Sciences, Seoul National University, Seoul 08826, Republic of Korea}
\email{dltpcks1@snu.ac.kr}

\thanks{Minhyun Kim gratefully acknowledge financial support by the German Research Foundation (GRK 2235 - 282638148). The research of Ki-Ahm Lee is supported by the National Research Foundation of Korea (NRF) grant funded by the Korea government (MSIP): NRF-2021R1A4A1027378.}

\begin{document}

\begin{abstract}
We study the boundary behavior of solutions to the Dirichlet problems for integro-differential operators with order of differentiability $s \in (0, 1)$ and summability $p>1$. We establish a nonlocal counterpart of the Wiener criterion, which characterizes a regular boundary point in terms of the nonlocal nonlinear potential theory.
\end{abstract}

\maketitle


\section{Introduction} \label{sec:introduction}


Wiener, in his pioneering works \cite{Wie24a,Wie24b}, provided a necessary and sufficient condition, so-called the Wiener criterion, for a boundary point to be regular in the case of the Laplacian. In the fundamental work by Littman, Stampacchia, and Weinberger \cite{LSW63}, the validity of the Wiener criterion was verified for second-order uniformly elliptic linear operators with bounded and measurable coefficients. After these works, the Wiener criterion has been extended to a large class of quasilinear operators of second order by using the nonlinear potential theory. See Maz'ya \cite{Maz70}, Gariepy and Ziemer \cite{GZ77}, Lindqvist and Martio \cite{LM85}, and Kilpel\"ainen and Mal\'y \cite{KM94}. We refer the reader to the books \cite{HKM06,MZ97} and references therein. Further results can be found in \cite{Lab02} for $k$-Hessian operators, \cite{AK04} for $p(x)$-Laplacian, and \cite{LL21} for operators with Orlicz growth.

For fractional-order operators, several sufficient conditions for a boundary point to be regular have been proposed. Under the exterior sphere condition, solutions to the Dirichlet problems are continuous across the boundary of a domain for the fractional Laplacian \cite{ROS14} and the fractional $p$-Laplacian \cite{IMS16,LL17}. See also \cite{ROS17}. Moreover, it is known that the measure density condition of the complement of a domain is sufficient for more general linear and nonlinear nonlocal operators with measurable coefficients \cite{KKP16,KKP17,LZLH20}. We point out that these conditions serve as sufficient conditions even for a wide class of local operators. However, they are far from sharp in the sense that they cannot serve as necessary conditions even though it is easier to check whether a domain satisfies these conditions than the Wiener criterion.

There is a relatively small amount of literature on the Wiener criterion for nonlocal operators. Eilertsen \cite{Eil00} proved the sufficiency of the Wiener criterion for solutions to the nonhomogeneous fractional Laplace equation with zero boundary data. Recently, Bj\"orn \cite{Bjo21} showed both sufficiency and necessity of the Wiener criterion for the fractional Laplacian by using Caffarelli--Silvestre extension \cite{CS07}. Note that the Wiener criterion in \cite{Bjo21} is obtained only for boundary data with compact support. Moreover, the methods in \cite{Bjo21,Eil00} are restricted to the case of the fractional Laplacian.

The aim of this paper is to establish a nonlocal counterpart of the Wiener criterion in full generality for a class of nonlinear nonlocal elliptic operators with measurable coefficients. Our method is based on the nonlocal nonlinear potential theory, which has been actively studied in recent years.

Let $n \in \mathbb{N}$, $s \in (0,1)$, $p>1$, and $\Lambda \geq 1$. We consider a nonlocal operator $\mathcal{L}$ defined by
\begin{equation*}
\mathcal{L}u(x) = 2 \, \mathrm{p.v.} \int_{\mathbb{R}^n} |u(x)-u(y)|^{p-2}(u(x)-u(y)) k(x,y) \,\mathrm{d}y,
\end{equation*}
where $k: \mathbb{R}^n \times \mathbb{R}^n \to [0, \infty]$ is a measurable function satisfying $k(x,y) = k(y,x)$ and
\begin{equation} \label{eq:ellipticity}
\frac{\Lambda^{-1}}{|x-y|^{n+sp}} \leq k(x,y) \leq \frac{\Lambda}{|x-y|^{n+sp}}.
\end{equation}
This operator is modeled on the fractional $p$-Laplacian, $(-\Delta)_p^s$, in which case the kernel is given by $k(x,y) = |x-y|^{-n-sp}$.

Let $\Omega$ be an open and bounded subset in $\mathbb{R}^n$. We say that a boundary point $x_0$ of $\Omega$ is {\it regular with respect to $\mathcal{L}$} if, for each function $g \in V^{s, p}(\Omega|\mathbb{R}^n) \cap C(\mathbb{R}^n)$, the unique $\mathcal{L}$-harmonic function $u \in V^{s, p}(\Omega|\mathbb{R}^n) \cap C(\Omega)$ with $u-g \in W^{s, p}_0(\Omega)$ satisfies
\begin{equation*}
\lim_{\Omega \ni x \to x_0}u(x)=g(x_0).
\end{equation*}
See \Cref{sec:weak-soln} for the definitions of function spaces and $\mathcal{L}$-harmonic functions, and \Cref{sec:harmonic} for the solvability of the Dirichlet problem in this function space. Note that a boundary data $g$ is not necessarily bounded nor compactly supported.

The following is the main theorem of this paper, which answers the question raised by \cite[Open Problem 7]{Pal18}.

\begin{theorem} \label{thm:main}
A boundary point $x_0 \in \partial \Omega$ is regular with respect to $\mathcal{L}$ if and only if
\begin{equation} \label{eq:wiener-int}
\int_0 \left( \frac{\mathrm{cap}_{s, p}(\overline{B_\rho(x_0)} \setminus \Omega, B_{2\rho}(x_0))}{\rho^{n-sp}} \right)^{\frac{1}{p-1}} \frac{\mathrm{d}\rho}{\rho} = + \infty.
\end{equation}
\end{theorem}

We refer to \Cref{sec:potential} for the definition of $(s,p)$-capacity. The integral in \eqref{eq:wiener-int} is called the {\it Wiener integral}. We remark that \Cref{thm:main} trivially holds when $p > n/s$. Indeed, if $p > n/s$, then it is easy to see that every boundary point is regular with respect to $\mathcal{L}$ by the Sobolev embedding theorem and that \eqref{eq:wiener-int} holds by \Cref{lem:cap}.

One can easily check that if a domain satisfies the exterior cone condition, the exterior $(\delta, R)$-Reifenberg flat condition, or the exterior corkscrew condition at $x_0 \in \partial \Omega$, then the Wiener integral at $x_0$ diverges. For the definitions of Reifenberg flat domain and corkscrew domain, we refer to \cite{JK82,Rei60}. As a direct consequence of \Cref{thm:main}, these geometric conditions are sufficient for $x_0$ to be regular.

\Cref{thm:main} implies that the regularity of a boundary point $x_0 \in \partial \Omega$ is completely determined by the local geometry of $\partial \Omega$. This result seems somewhat surprising because we deal with nonlocal equations. Indeed, in the intermediate steps of the proof of \Cref{thm:main}, we make substantial use of various integral estimates such as local boundedness, weak Harnack inequality, and Wolff potential estimate that contain a nonlocal tail term; see \Cref{sec:bdry-est} for details. Nevertheless, at the final stage, we observe that the long-range interactions of the solution disappear in the Wiener criterion. One possible interpretation of this phenomenon is that the information concerning whether the Wiener integral diverges or not is strong enough to absorb the nonlocal effect coming from the corresponding tail terms.

The following is an immediate corollary of \Cref{thm:main}, which is expected as in the case of local operators.
\begin{corollary}
The regularity of a boundary point depends only on $n$, $s$, and $p$, not on the operator $\mathcal{L}$ itself.
\end{corollary}

Let us outline the strategy of the proof of \Cref{thm:main}. The starting point of the sufficient part is to capture the local behaviors of weak solutions near the boundary. More precisely, we establish the local boundedness for weak subsolutions and the weak Harnack inequality for weak supersolutions up to the boundary. These are boundary variants of the interior estimates provided in \cite{BP16,DCKP14,DCKP16,Kas09}. To this end, we modify localization techniques developed in the aforementioned papers to obtain Caccioppoli-type estimates near the boundary and then use Moser's iteration.

The essential challenge in the nonlocal framework arises from the nonlocal tail term, which is a quantity encoding the long-range interactions; see \eqref{eq:tail} for the precise definition. Indeed, after repeated applications of the Caccioppoli-type inequalities and the weak Harnack inequality, we control the capacity of the complement of the domain by the sum of the infimum value of supersolutions and the nonlocal tail term. The nonlocal effect turns out to be negligible from the sufficient condition that the Wiener integral diverges. In this step, we divide the domain into a countable union of annuli and then analyze the quantified tail term in an iterative way.

Let us move on to the necessary part of the proof. In this case, the nonlocal nonlinear potential theory is more involved since the main ingredient is the pointwise estimate of solutions in terms of the Wolff potential. This estimate, together with additional regularity results, was already obtained by Kuusi, Mingione, and Sire \cite{KMS15} for SOLA (Solutions Obtained as Limits of Approximations). Nevertheless, we could not directly adopt the estimate since they imposed the natural restriction $p>2-\frac{s}{n}$ to guarantee the existence of SOLA for a given Radon measure. This problem can be overcome by using a different type of solution, namely, $\mathcal{L}$-superharmonic function, as in the local case (see \cite{KM92,KM94}). In this way, we cover all ranges of $p >1$. However, we still need to prove the Wolff potential estimate for $\mathcal{L}$-superharmonic functions for $p >1$. Therefore, we present a new proof of the Wolff potential estimate, mainly following the iteration scheme given in \cite{KM94}. Again, one has to take into account a suitable tail, which requires further effort to handle. The necessary part is then concluded by using the relationship between the Wolff potential generated by the $\mathcal{L}$-potential and the $\mathcal{L}$-distribution; see \Cref{sec:necessary} for details.

This paper is organized as follows. In \Cref{sec:preliminaries}, we introduce several definitions and collect preliminary results related to our main theorem. \Cref{sec:bdry-est} consists of local estimates up to the boundary, namely, the local boundedness and the weak Harnack inequality. The proofs for the sufficiency and the necessity of the Wiener criterion are provided in 
\Cref{sec:sufficient} and \Cref{sec:necessary}, respectively. Some algebraic inequalities used in the previous sections can be found in \Cref{sec:inequalities}.


\section{Preliminaries} \label{sec:preliminaries}


In this section, we recall several function spaces to define weak solutions and harmonic functions with respect to $\mathcal{L}$. By using harmonic functions, we define a regular boundary point with respect to $\mathcal{L}$. Moreover, we introduce the $\mathcal{L}$-potential, which is closely related to the $(s, p)$-capacity.


\subsection{Weak solution} \label{sec:weak-soln}


Let us first collect definitions of fractional Sobolev spaces and tail spaces. For $s \in (0,1)$, $p \in [1, \infty)$, and an open set $\Omega \subset \mathbb{R}^n$, let $W^{s, p}(\Omega)$ be the usual {\it fractional Sobolev space} with the norm
\begin{equation*}
\|u\|_{W^{s, p}(\Omega)} = \left( \|u\|_{L^p(\Omega)}^p + [u]_{W^{s, p}(\Omega)}^p \right)^{1/p} = \left( \int_{\Omega} |u(x)|^p \,\mathrm{d}x + \int_{\Omega} \int_{\Omega} \frac{|u(x)-u(y)|^p}{|x-y|^{n+sp}} \,\mathrm{d}y \,\mathrm{d}x \right)^{1/p},
\end{equation*}
see \cite{DNPV12}. By $W^{s, p}_{\mathrm{loc}}(\Omega)$ we denote the space of functions that belong to $W^{s, p}(\Omega')$ for each open set $\Omega' \subset\subset \Omega$. We recall the {\it tail space} $L^{p-1}_{sp}(\mathbb{R}^n)$, which is given by
\begin{equation*}
L^{p-1}_{sp}(\mathbb{R}^n) = \left\lbrace u \in L^{p-1}_{\mathrm{loc}}(\mathbb{R}^n): \int_{\mathbb{R}^n} \frac{|u(y)|^{p-1}}{(1+|y|)^{n+sp}} \,\mathrm{d}y < +\infty \right\rbrace.
\end{equation*}
It is well known and easy to check that $W^{s, p}(\mathbb{R}^n) \subset L^{p-1}_{sp}(\mathbb{R}^n)$ and that the {\it nonlocal tail} (or {\it tail} for short) $\mathrm{Tail}(u; x_0, r)$, for any $x_0 \in \mathbb{R}^n$ and $r > 0$, is well defined if $u \in L^{p-1}_{sp}(\mathbb{R}^n)$, where
\begin{equation} \label{eq:tail}
\mathrm{Tail}(u; x_0, r) = \left( r^{sp} \int_{\mathbb{R}^n \setminus B_r(x_0)} \frac{|u(y)|^{p-1}}{|y-x_0|^{n+sp}} \,\mathrm{d}y \right)^{1/(p-1)}.
\end{equation}
The nonlocal tail and the tail space have been introduced and studied in a systematic way in \cite{DCKP14,DCKP16,Kas11} and \cite{Coz17,IMS16,KKL19,KKP16,KKP17,KMS15}.

Let $\Lambda \geq 1$ and $k$ be a measurable function that is symmetric and satisfies \eqref{eq:ellipticity}. For measurable functions $u, v: \mathbb{R}^n \to \mathbb{R}$, we define a quantity
\begin{equation*}
\mathcal{E}^k(u,v)=\int_{\mathbb{R}^n} \int_{\mathbb{R}^n} |u(x)-u(y)|^{p-2} (u(x)-u(y))(v(x)-v(y)) k(x, y) \,\mathrm{d}y\,\mathrm{d}x,
\end{equation*}
provided that it is finite. In particular, we denote by $\mathcal{E}^{s, p}(u, v)=\mathcal{E}^k(u, v)$ when $k$ is given by $k(x, y) = |x-y|^{-n-sp}$. We also write $\mathcal{E}^{k}(u)=\mathcal{E}^{k}(u, u)$ and $\mathcal{E}^{s, p}(u) = \mathcal{E}^{s, p}(u, u) = [u]_{W^{s, p}(\mathbb{R}^n)}^p$. Note that, if $u \in W^{s, p}_{\mathrm{loc}}(\Omega) \cap L^{p-1}_{sp}(\mathbb{R}^n)$ and $v \in C_c^{\infty}(\Omega)$, then both $\mathcal{E}^k(u, v)$ and $\mathcal{E}^{s, p}(u, v)$ are finite. We now define a weak subsolution, supersolution, and solution with respect to $\mathcal{L}$. See also \cite{KKL19,KKP16,KKP17,LL17,Pal18}.

\begin{definition} \label{def:weak-soln}
Let $f$ be a measurable function in $\Omega$. A function $u \in W^{s, p}_{\mathrm{loc}}(\Omega)$ with $u_- \in L^{p-1}_{sp}(\mathbb{R}^n)$ is called a {\it weak supersolution of $\mathcal{L}u=f$ in $\Omega$} if $\langle f, \varphi \rangle:= \int_{\Omega}f(x) \varphi(x) \,\mathrm{d}x$ is well defined and
\begin{equation} \label{eq:weak-supersoln}
\mathcal{E}^k(u, \varphi) \geq \langle f, \varphi \rangle
\end{equation}
for all nonnegative functions $\varphi \in C_c^{\infty}(\Omega)$. A function $u$ is called a {\it weak subsolution} if $-u$ is a weak supersolution. We call $u$ a {\it weak solution} if it is both a weak subsolution and a weak supersolution.
\end{definition}

Note that the summability assumption $u_- \in L^{p-1}_{sp}(\mathbb{R}^n)$ for weak supersolution in \Cref{def:weak-soln} can be replaced by $u \in L^{p-1}_{sp}(\mathbb{R}^n)$. It is proved \cite[Lemma 1]{KKP17} that these two definitions are actually equivalent. 

Since test functions are not always in $C_c^{\infty}(\Omega)$, it is desirable to deal with a larger class of test functions than $C_c^{\infty}(\Omega)$. For this purpose, we consider the Banach spaces
\begin{equation} \label{eq:space-V}
V^{s, p}(\Omega|\mathbb{R}^n) := \left\lbrace u: \mathbb{R}^n \to \mathbb{R} : u|_{\Omega} \in L^p(\Omega), \frac{|u(x)-u(y)|}{|x-y|^{n/p+s}} \in L^p(\Omega \times \mathbb{R}^n) \right\rbrace,
\end{equation}
equipped with the norm
\begin{equation*}
\begin{split}
\|u\|_{V^{s, p}(\Omega|\mathbb{R}^n)} 
&= \left( \|u\|_{L^p(\Omega)}^p + [u]_{V^{s, p}(\Omega|\mathbb{R}^n)}^p \right)^{1/p} \\
&= \left( \int_{\Omega} |u(x)|^p \,\mathrm{d}x + \int_{\Omega} \int_{\mathbb{R}^n} \frac{|u(x)-u(y)|^p}{|x-y|^{n+sp}} \,\mathrm{d}y \,\mathrm{d}x \right)^{1/p},
\end{split}
\end{equation*}
and
\begin{equation} \label{eq:space-W}
W^{s, p}_0(\Omega) := \overline{C_c^{\infty}(\Omega)}^{V^{s, p}(\Omega|\mathbb{R}^n)} = \overline{C_c^{\infty}(\Omega)}^{W^{s, p}(\mathbb{R}^n)}.
\end{equation}
Note that $V^{s, p}(\Omega|\mathbb{R}^n) \subset W^{s, p}_{\mathrm{loc}}(\Omega) \cap L^{p-1}_{sp}(\mathbb{R}^n)$. We refer the reader to \cite{FG20} for more properties of these spaces.

\begin{remark}
\begin{enumerate}[(i)]
\item
The notation $V$ in the definition \eqref{eq:space-V} is adopted from \cite{DK20,FKV15,FGKV20}. Note that some authors denote the space in \eqref{eq:space-V} by $W^{s, p}(\Omega|\mathbb{R}^n)$ or $\mathbb{W}^{s,p}(\Omega)$. See, e.g., \cite{Coz17,FG20}.
\item
In this paper, we use $W^{s, p}_0(\Omega)$ to denote the closure of $C_c^{\infty}(\Omega)$ in $V^{s, p}(\Omega|\mathbb{R}^n)$ as in \cite{DK13,FKV15,LL17}. However, some authors often define $W^{s, p}_0(\Omega)$ to be the closure of $C_c^{\infty}(\Omega)$ in $W^{s, p}(\Omega)$, see \cite{DNPV12,HT08}. Moreover, the space
\begin{equation*}
W^{s, p}_{\Omega}(\mathbb{R}^n)=\lbrace u \in W^{s, p}(\mathbb{R}^n): u=0 \text{ a.e. on } \mathbb{R}^n \setminus \Omega \rbrace,
\end{equation*}
which contains \eqref{eq:space-W} by definition, is also used in \cite{FKV15} as a space of test functions. One should be careful since these three spaces are not equal in general especially on non-smooth domains. Indeed, if $\Omega$ is an open set with continuous boundary, then $W^{s, p}_0(\Omega) = W^{s, p}_\Omega(\mathbb{R}^n)$, but there exists a domain $\Omega$ such that $W^{s, p}_0(\Omega) \subsetneq W^{s, p}_{\Omega}(\mathbb{R}^n)$, see \cite{FSV15}. Note that it is known \cite{Gri11,HT08} that if $\Omega$ is a bounded open set with a Lipschitz boundary and $s \neq 1/p$, then all three spaces coincide, i.e.,
\begin{equation*}
W^{s, p}_0(\Omega) = W^{s, p}_\Omega(\mathbb{R}^n) = \overline{C_c^{\infty}(\Omega)}^{W^{s, p}(\Omega)}.
\end{equation*}
However, we emphasize that we do not assume any regularity on $\Omega$ and that we will only make use of the space \eqref{eq:space-W} in this paper.
\end{enumerate}
\end{remark}

If we know a priori that a weak supersolution has an additional regularity, then the space of test functions can be extended by a standard approximation argument.

\begin{proposition}
Let $f \in (W^{s, p}_0(\Omega))^{\ast}$. If $u \in V^{s, p}(\Omega|\mathbb{R}^n)$ is a weak supersolution of $\mathcal{L}u=f$ in $\Omega$, then \eqref{eq:weak-supersoln} holds for all nonnegative functions $\varphi \in W^{s, p}_0(\Omega)$.
\end{proposition}

The next theorem shows that the Dirichlet problem for the equation $\mathcal{L}u=0$ with Sobolev boundary data is solvable.

\begin{theorem} \label{thm:DP}
Suppose that $\Omega$ is bounded and that $f \in (W^{s, p}_0(\Omega))^{\ast}$, $g \in V^{s, p}(\Omega|\mathbb{R}^n)$. Then, there exists a unique weak solution $u \in V^{s, p}(\Omega|\mathbb{R}^n)$ of $\mathcal{L}u=f$ in $\Omega$ with $u-g\in W^{s, p}_0(\Omega)$.
\end{theorem}

In \Cref{thm:DP}, we understand the boundary data in the weak sense, namely, $u=g$ in $\mathbb{R}^n \setminus \Omega$ in the sense of $V^{s, p}(\Omega|\mathbb{R}^n)$ if and only if $u-g \in W^{s, p}_0(\Omega)$. We generalize this concept to inequalities in the following way. We say that {\it $u\leq 0$ in $\mathbb{R}^n \setminus \Omega$ in the sense of $V^{s, p}(\Omega|\mathbb{R}^n)$} if $u_+:=\max \lbrace u, 0 \rbrace \in W^{s, p}_0(\Omega)$. Other types of inequalities can be defined in a similar manner. For example, $u \geq 0$ ($u \leq v$, respectively) in the sense of $V^{s, p}(\Omega|\mathbb{R}^n)$ if $-u \leq 0$ ($u-v \leq 0$, respectively) in the same sense. 

The uniqueness part of \Cref{thm:DP} follows from the comparison principle.

\begin{theorem} [Comparison principle] \label{thm:comparison}
Suppose that $\Omega$ is bounded. Let $u, v \in V^{s, p}(\Omega|\mathbb{R}^n)$ be such that $u \geq v$ in $\mathbb{R}^n \setminus \Omega$ in the sense of $V^{s, p}(\Omega|\mathbb{R}^n)$ and
\begin{equation*}
\mathcal{E}^k(u, \varphi) \geq \mathcal{E}^k(v, \varphi)
\end{equation*}
for all nonnegative functions $\varphi \in W^{s, p}_0(\Omega)$, then $u \geq v$ a.e. in $\mathbb{R}^n$.
\end{theorem}

\Cref{thm:DP} and \Cref{thm:comparison} have been proved in slightly different settings in the literature \cite{FKV15,FG20,LL17}, but the same arguments work in our setting.

Let us finish the section with some functional inequalities. See \cite[Theorem 6.7]{DNPV12}, \cite[Corollary 2.1]{Pon04}, and \cite[Proposition 3.2]{Sal20} for the proofs.

\begin{theorem}[Fractional Sobolev inequality]
Let $p \in (1, n/s)$ and $R > 0$. There exists a constant $C = C(n, s, p) > 0$ such that
\begin{equation*}
\left( \int_{B_R} |u(x)|^{p^{\ast}} \,\mathrm{d}x \right)^{p/p^{\ast}} \leq C \left( \int_{B_R} \int_{B_R} \frac{|u(x)-u(y)|^p}{|x-y|^{n+sp}} \,\mathrm{d}y \,\mathrm{d}x + R^{-sp} \int_{B_R} |u(x)|^p \,\mathrm{d}x \right)
\end{equation*}
for all $u \in W^{s, p}(B_R)$, where $p^{\ast} = np/(n-sp)$.
\end{theorem}

\begin{theorem}[Fractional Poincar\'e inequality I] \label{poincare1}
	Let $p \in (1, \infty)$ and $R > 0$. There exists a constant $C = C(n, s, p) > 0$ such that
	\begin{align*}
	\int_{B_R} |u(x) - (u)_{B_R}|^p \,\mathrm{d}x \leq C R^{sp} \int_{B_R} \int_{B_R} \frac{|u(x)-u(y)|^p}{|x-y|^{n+sp}} \,\mathrm{d}y \,\mathrm{d}x
	\end{align*}
	for all $u \in W^{s,p}(B_R)$, where $(u)_{B_R} = \fint_{B_R} u(x) \,\mathrm{d}x$.
\end{theorem}

\begin{theorem}[Fractional Poincar\'e inequality II]\label{poincare}
	Let $p \in (1, \infty)$ and $\Omega \subset \mathbb{R}^n$ be open and bounded. There exists a constant $C = C(n, s, p) > 0$ such that
	\begin{align*}
	\int_{\Omega} |u(x)|^p \,\mathrm{d}x \leq C (\mathrm{diam}(\Omega))^{sp} \int_{\mathbb{R}^n} \int_{\mathbb{R}^n} \frac{|u(x)-u(y)|^p}{|x-y|^{n+sp}} \,\mathrm{d}y \,\mathrm{d}x
	\end{align*}
	for all $u \in W_0^{s,p}(\Omega)$.
\end{theorem}


\subsection{Harmonic function} \label{sec:harmonic}


Let us introduce a harmonic function with respect to $\mathcal{L}$ and define a regular boundary point with respect to $\mathcal{L}$. We first recall that a weak solution of the equation $\mathcal{L}u = 0$ can be redefined in a set of measure zero so that it is continuous.

\begin{theorem} \cite[Theorem 9]{KKP17} \label{thm:regularization}
Let $u$ be a weak supersolution of $\mathcal{L}u=0$ in $\Omega$. Then
\begin{equation*}
u(x) = \mathrm{ess} \liminf_{y \to x} u(y)
\end{equation*}
for a.e. $x \in \Omega$. In particular, $u$ has a representative that is lower semicontinuous in $\Omega$.
\end{theorem}

Let us next provide a definition of $\mathcal{L}$-harmonic function. See also \cite{KKL19,KKP17,Pal18}. We will define $\mathcal{L}$-superharmonic and $\mathcal{L}$-subharmonic functions later.

\begin{definition}
A function $u$ is said to be {\it $\mathcal{L}$-harmonic in $\Omega$} if it is a weak solution of $\mathcal{L}u=0$ in $\Omega$ that is continuous in $\Omega$.
\end{definition}

By \Cref{thm:DP} and \Cref{thm:regularization}, we know that there exists a unique $\mathcal{L}$-harmonic function with Sobolev boundary data.

\begin{definition} [Regular point] \label{def:regular}
A boundary point $x_0 \in \partial \Omega$ is said to be {\it regular (with respect to $\mathcal{L}$)} if, for each function $g \in V^{s, p}(\Omega|\mathbb{R}^n) \cap C(\mathbb{R}^n)$, the unique $\mathcal{L}$-harmonic function $u \in V^{s, p}(\Omega|\mathbb{R}^n) \cap C(\Omega)$ with $u-g \in W^{s, p}_0(\Omega)$ satisfies
\begin{equation} \label{eq:limit}
\lim_{\Omega \ni x \to x_0}u(x)=g(x_0).
\end{equation}
\end{definition}

\Cref{def:regular} corresponds to the definition of Sobolev regular boundary point \cite[Chapter 6]{HKM06} in the case of local operators. Another definition of regular boundary point, which is based on Perron's solutions for the fractional $p$-Laplacian, is provided in \cite{LL17}. This definition corresponds to that of regular boundary point \cite[Chapter 9]{HKM06} for local operators. We do not deal with Perron's solutions for the operator $\mathcal{L}$ within the scope of this paper, but $\mathcal{L}$-superharmonic and $\mathcal{L}$-subharmonic functions are useful for the theory.

\begin{definition} \label{def:superharmonic}
A function $u: \mathbb{R}^n \to (-\infty, \infty]$ is said to be {\it $\mathcal{L}$-superharmonic in $\Omega$} if it satisfies the following properties:
\begin{enumerate}[(i)]
\item
$u < +\infty$ almost everywhere in $\mathbb{R}^n$,
\item
$u$ is lower semicontinuous in $\Omega$,
\item
for each open set $\Omega' \subset\subset \Omega$ and each weak solution $v \in C(\overline{\Omega'})$ of $\mathcal{L}v=0$ in $\Omega'$ with $v_+ \in L^{\infty}(\mathbb{R}^n)$ such that $u \geq v$ on $\partial \Omega'$ and a.e. on $\mathbb{R}^n \setminus \Omega'$, it holds that $u \geq v$ in $\Omega'$,
\item
$u_- \in L^{p-1}_{sp}(\mathbb{R}^n)$.
\end{enumerate}
A function $u$ is said to be {\it $\mathcal{L}$-subharmonic in $\Omega$} if $-u$ is $\mathcal{L}$-superharmonic in $\Omega$.
\end{definition}

We collect some results on $\mathcal{L}$-superharmonic functions that we need in the sequel.

\begin{theorem} \cite[Theorem 1]{KKP17} \label{thm:KKP17}
If $u$ is an $\mathcal{L}$-superharmonic function in $\Omega$, then $u(x) = \liminf_{y \to x} u(y) = \mathrm{ess} \,\liminf_{y\to x} u(y)$ for every $x \in \Omega$. If, in addition, $u$ is locally bounded in $\Omega$ or $u \in W^{s, p}_{\mathrm{loc}}(\Omega)$, then it is a weak supersolution in $\Omega$.
\end{theorem}

\begin{remark} \label{rmk:supersoln}
\begin{enumerate}[(i)]
\item
\cite[Corollary 7]{KKP17} A function $u$ is $\mathcal{L}$-harmonic in $\Omega$ if and only if $u$ is both $\mathcal{L}$-superharmonic and $\mathcal{L}$-subharmonic in $\Omega$.
\item
\cite[Theorem 12]{KKP17} If $u$ is a lower semicontinuous weak supersolution of $\mathcal{L}u=0$ in $\Omega$ satisfying
\begin{equation*}
u(x) = \mathrm{ess} \liminf_{y\to x} u(y)
\end{equation*}
for every $x \in \Omega$, then $u$ is $\mathcal{L}$-superharmonic in $\Omega$.
\end{enumerate}
\end{remark}


\subsection{Potential theory} \label{sec:potential}


By using $\mathcal{L}$-harmonic functions, we define the $\mathcal{L}$-potential, which will play a crucial role in the development of the theory.

\begin{definition}
Let $\Omega$ be a bounded open set and $K \subset \Omega$ be a compact set. Let $\psi: \mathbb{R}^n \to [0,1]$ be such that $\psi \in C_c^{\infty}(\Omega)$ and $\psi = 1$ on $K$. The $\mathcal{L}$-harmonic function in $\Omega \setminus K$ with $u-\psi \in W^{s, p}_0(\Omega\setminus K)$ is called the {\it $\mathcal{L}$-potential of $K$ in $\Omega$} and denoted by $\mathfrak{R}(K, \Omega)$.
\end{definition}

By \Cref{thm:comparison} it is immediate that the definition of $\mathfrak{R}(K, \Omega)$ is independent of the particular choice of $\psi$. 

We next recall the definition of the $(s, p)$-capacity, which is one of the fundamental quantities in the potential theory, see \cite{Bjo21,HKM06,Maz79}. Suppose that $K$ is a compact subset of $\Omega$ and let
\begin{equation*}
W(K, \Omega) := \lbrace v \in C_c^{\infty}(\Omega): v \geq 1 \text{ on } K \rbrace.
\end{equation*}
The {\it $(s, p)$-capacity of $K$ with respect to $\Omega$} is defined by
\begin{equation} \label{eq:def-capacity}
\mathrm{cap}_{s, p}(K, \Omega) = \inf_{v \in W(K, \Omega)} \mathcal{E}^{s, p}(v) = \inf_{v \in W(K, \Omega)} \int_{\mathbb{R}^n} \int_{\mathbb{R}^n} \frac{|v(x)-v(y)|^p}{|x-y|^{n+sp}} \,\mathrm{d}y \,\mathrm{d}x.
\end{equation}
By definition, $\mathrm{cap}_{s, p}(\cdot, \cdot)$ is increasing in the first argument and decreasing in the second argument with respect to the set inclusion. Moreover, by using an approximation, the set $W(K, \Omega)$ in the definition \eqref{eq:def-capacity} can be replaced by a larger set
\begin{equation*}
W_0(K, \Omega) := \overline{W(K, \Omega)}^{W^{s, p}(\mathbb{R}^n)}.
\end{equation*}
A function in $W_0(K, \Omega)$ is said to be {\it admissible for $\mathrm{cap}_{s, p}(K, \Omega)$}.

We now would like to propose useful properties of the $\mathcal{L}$-potential. For this purpose, we require a generalized notion of $(s,p)$-capacity, namely,
\begin{align*}
	\mathcal{C}^k(K, \Omega):=\inf_{v \in W(K, \Omega)} \mathcal{E}^{k}(v).
\end{align*}
Note that if $k(x,y)=|x-y|^{-n-sp}$, then $\mathcal{C}^k(K, \Omega)=\mathrm{cap}_{s,p}(K, \Omega)$.

\begin{lemma}\label{lem:potential}
	Suppose that $\Omega$ is open and bounded, and $K$ is a compact subset of $\Omega$.
	\begin{enumerate}[(i)]
		\item $\mathfrak{R}(K, \Omega)$ is admissible for $\mathrm{cap}_{s,p}(K, \Omega)$. In particular, $\mathfrak{R}(K, \Omega) \in W_0^{s, p}(\Omega)$.
		
		\item Let
		\begin{equation*}
			\widetilde{W}(K, \Omega) := \lbrace v \in C_c^{\infty}(\Omega): v = 1 \text{ in a neighborhood of } K \rbrace.
		\end{equation*}
	Then, we have
	\begin{align*}
		\mathcal{C}^k(K, \Omega)=\inf_{v \in \widetilde{W}(K, \Omega)} \mathcal{E}^k(v).
	\end{align*}

	\item $\mathfrak{R}(K, \Omega)$ is a minimizer of $\mathcal{E}^k$ among all functions belonging to $W_0(K, \Omega)$, i.e.,
	\begin{align*}
		\mathcal{E}^k(\mathfrak{R}(K, \Omega))=\mathcal{C}^k(K, \Omega).
	\end{align*}

	\item $\mathfrak{R}(K, \Omega)$ is a weak supersolution of $\mathcal{L}u=0$ in $\Omega$.
	\end{enumerate}
	
\end{lemma}

\begin{proof}
	\begin{enumerate}[(i)]
		\item Let $\psi \in C_c^{\infty}(\Omega)$ be such that $0 \leq \psi \leq 1$ and $\psi =1$ on $K$. By the definition of $\mathcal{L}$-potential, there exists a sequence of functions $\{\varphi_j\}_{j=1}^{\infty} \subset C_c^{\infty}(\Omega \setminus K)$ such that 
		\begin{align*}
			\varphi_j \to \mathfrak{R}(K, \Omega)-\psi \quad \text{ in } W^{s, p}(\mathbb{R}^n).
		\end{align*}
	Since $\varphi_j+\psi \in W(K, \Omega)$, the desired result follows.
		
		\item Since $\widetilde{W}(K, \Omega) \subset W(K, \Omega)$, it is immediate that
		\begin{align*}
			\mathcal{C}^k(K, \Omega) \leq \inf_{v \in \widetilde{W}(K, \Omega)} \mathcal{E}^k(v).	
		\end{align*}
	For the converse inequality, given $\varepsilon>0$, let $\varphi \in W(K, \Omega)$ be such that 
	\begin{align*}
		\mathcal{E}^k(\varphi) \leq \mathcal{C}^k(K, \Omega)+\varepsilon.
	\end{align*}
	We denote by $\{\eta_j\}_{j=1}^{\infty}$ a sequence of functions belonging to $C^{\infty}(\mathbb{R})$ such that $0 \leq \eta_j(t) \leq 1$, $\eta_j(t)=0$ in a neighborhood of $(-\infty, 0]$, $\eta_j(t)=1$ in a neighborhood of $[1, \infty)$, and $\|\eta_j'\|_{\infty} \leq 1+1/j$ for all $j \in \mathbb{N}$. Then, since $\eta_j(\varphi(x)) \in \widetilde{W}(K, \Omega)$, we have
	\begin{align*}
		\inf_{v \in \widetilde{W}(K, \Omega)} \mathcal{E}^k(v) &\leq \int_{\mathbb{R}^n}\int_{\mathbb{R}^n} |\eta_j(\varphi(x))-\eta_j(\varphi(y))|^p k(x,y) \,\mathrm{d}y\,\mathrm{d}x \\
		&\leq \|\eta_j'\|_{\infty}^p\int_{\mathbb{R}^n}\int_{\mathbb{R}^n} |\varphi(x)-\varphi(y)|^p k(x,y) \,\mathrm{d}y\,\mathrm{d}x \\
		&\leq (1+1/j)^p \cdot (\mathcal{C}^k(K, \Omega)+\varepsilon).
	\end{align*}
By letting $j \to \infty$ and $\varepsilon \to 0$, we finish the proof.

	\item Since $W_0(K, \Omega)$ is the completion of $W(K, \Omega)$ in $W^{s, p}(\mathbb{R}^n)$, we have
	\begin{align*}
		\mathcal{C}^k(K, \Omega)=\inf_{v \in W_0(K, \Omega)} \mathcal{E}^{k}(v).
	\end{align*}
	Then, it follows from part (i) that 
	\begin{align*}
		\mathcal{C}^k(K, \Omega) \leq \mathcal{E}^{k}(\mathfrak{R}(K, \Omega)).
	\end{align*}
	For the reversed inequality, let $v \in \widetilde{W}(K, \Omega)$. Then, since $\mathfrak{R}(K, \Omega)-v \in W_0^{s,p}(\Omega \setminus K)$ and $\mathfrak{R}(K, \Omega)$ is $\mathcal{L}$-harmonic in $\Omega \setminus K$, we obtain
	\begin{align*}
		\mathcal{E}^k(\mathfrak{R}(K, \Omega), \mathfrak{R}(K, \Omega)-v)= 0.
	\end{align*}
	This, together with Young's inequality, implies that 
	\begin{align*}
		\mathcal{E}^k(\mathfrak{R}(K, \Omega))=\mathcal{E}^k(\mathfrak{R}(K, \Omega), v) \leq \frac{p-1}{p}\mathcal{E}^k(\mathfrak{R}(K, \Omega))+\frac{1}{p}\mathcal{E}^k(v),
	\end{align*}
	and hence $\mathcal{E}^k(\mathfrak{R}(K, \Omega)) \leq \mathcal{E}^k (v)$. By taking the infimum over $\widetilde{W}(K, \Omega)$ and recalling part (ii), the conclusion follows.
	
	\item From part (i), we know that $\mathfrak{R}(K, \Omega) \in W^{s,p}_{\mathrm{loc}}(\Omega) \cap L_{sp}^{p-1}(\mathbb{R}^n)$. Let $\varphi \in C_c^{\infty}(\Omega)$ be nonnegative in $\Omega$. Then for any $t \geq 0$, $\mathfrak{R}(K, \Omega)+t\varphi$ belongs to $W_0(K, \Omega)$. By applying part (iii), we have
	\begin{align*}
		\left.\frac{\mathrm{d}}{\mathrm{d}t} \right|_{t=0} \mathcal{E}^k(\mathfrak{R}(K, \Omega)+t\varphi) \geq 0,
	\end{align*}
	which leads to $\mathcal{E}^k(\mathfrak{R}(K, \Omega), \varphi) \geq 0$.
	\end{enumerate}
\end{proof}

Furthermore, we can estimate the $(s, p)$-capacity of a closed ball $\overline{B_r(x_0)}$ with respect to a concentric ball $B_R(x_0)$ as follows:
\begin{lemma}\label{lem:cap}
	Let $0<r \leq R/2$, then
	\begin{equation*}
	\mathrm{cap}_{s,p}(\overline{B_r(x_0)}, B_R(x_0)) \simeq 
	\begin{cases}
	r^{n-sp}  &\text{if } n>sp, \\
	R^{n-sp} &\text{if } n<sp, \\
	(\log(R/r))^{1-p}  &\text{if } n=sp,
	\end{cases}
	\end{equation*}
	where the comparable constants depend only on $n$, $s$, and $p$.
\end{lemma}

\begin{proof}
We recall from \cite[Lemma 2.4]{Bjo21} that
	\begin{align*}
	\mathrm{cap}_{p, |t|^a}(B_{r/2}(z_0), B_R(z_0)) \lesssim \mathrm{cap}_{s,p}(\overline{B_r(x_0)}, B_{R}(x_0)) \lesssim \mathrm{cap}_{p, |t|^a} (B_r(z_0), B_R(z_0)),
	\end{align*}
	where $a=p-1-sp$ and $z_0=(x_0, 0) \in \mathbb{R}^{n+1}$. See \cite{Bjo21,HKM06} for the precise definition of the weighted variational capacity $\mathrm{cap}_{p, |t|^a}$. For the upper bound, we apply \cite[Theorem 2.19]{HKM06} to conclude that
	\begin{align*}
	\mathrm{cap}_{p, |t|^a} (B_r(z_0), B_R(z_0)) 
	&\lesssim \left(\int_{B_R(z_0) \setminus B_r(z_0)} |z-z_0|^{np/(1-p)} |t|^{a/(1-p)} \,\mathrm{d}x \,\mathrm{d}t \right)^{1-p}\\
	&\simeq \left( \int_r^R \rho^{\frac{sp-n}{p-1}-1} \,\mathrm{d}\rho \right)^{1-p}\\
	&\simeq
	\begin{cases}
	r^{n-sp}  &\text{if } n>sp, \\
	R^{n-sp} &\text{if } n<sp, \\
	(\log(R/r))^{1-p}  &\text{if } n=sp,
	\end{cases}
	\end{align*}
	where $z=(x,t)$. The lower bound follows from \cite[Theorem 2.18]{HKM06}.
\end{proof}


\section{Local estimates up to the boundary} \label{sec:bdry-est}


To prove the Wiener criterion, it is crucial to study the behavior of weak solutions to $\mathcal{L}u=0$ near the boundary. In this section, we establish local estimates of weak solutions up to the boundary. More precisely, we study the local boundedness for weak subsolutions and the weak Harnack inequality for weak supersolutions up to the boundary by using Moser's iteration.


\subsection{Caccioppoli-type estimates up to the boundary} \label{sec:Caccioppoli}


The first step towards local estimates for weak solutions up to the boundary is to establish Caccioppoli-type estimates up to the boundary. For this purpose, we generalize the definition of inequality of $V^{s, p}(\Omega|\mathbb{R}^n)$ functions in the following way. Let $T$ be any set in $\mathbb{R}^n$ and $u$ be a $V^{s, p}(\Omega|\mathbb{R}^n)$ function. We say that {\it $u \leq 0$ on $T$ in the sense of $V^{s, p}(\Omega|\mathbb{R}^n)$} if $u_+ := \max \lbrace u, 0\rbrace$ is the limit in $V^{s, p}(\Omega|\mathbb{R}^n)$ of a sequence of $C^{0,1}(\mathbb{R}^n)$-functions with compact support in $\mathbb{R}^n \setminus T$. In a similar manner, we say that  $u \geq 0$ ($u \leq v$, respectively) {\it on $T$ in the sense of $V^{s, p}(\Omega|\mathbb{R}^n)$} if $-u \leq 0$ ($u-v \leq 0$, respectively) in the same sense. In particular, if $T = \mathbb{R}^n \setminus \Omega$, this definition coincides with the earlier one given in \Cref{sec:weak-soln}. We also define
\begin{equation*}
\wsup_{T} u = \inf \lbrace k \in \mathbb{R}: u \leq k \text{ on } T \text{ in the sense of } V^{s, p}(\Omega|\mathbb{R}^n) \rbrace
\end{equation*}
and $\winf_{T} u = - \wsup_{T}(-u)$.

We prove the following three lemmas, one for weak subsolutions and the others for weak supersolutions, at the same time as in the case of local operators, see \cite[Theorem 8.25 and 8.26]{GT01}. In what follows, we let $\beta, \gamma \in \mathbb{R}$ be such that $\gamma=\beta+p-1\neq 0$, and define quantities
\begin{equation} \label{eq:modular}
\Phi_{\gamma; \, x_0, R}(u) = \left( \fint_{B_R(x_0)} |u(x)|^{\gamma} \,\mathrm{d}x \right)^{1/\gamma},
\end{equation}
where $\fint_A$ denotes $|A|^{-1} \int_A$. Moreover, throughout the paper, we will denote a universal constant by $C > 0$, which may be different from line to line.

\begin{lemma} \label{lem:Caccio-subsoln}
Let $p \in (1, \infty)$ and assume $\gamma = \beta+p-1 > p-1$. If $u \in V^{s, p}(\Omega|\mathbb{R}^n)$ is a bounded weak subsolution of $\mathcal{L}u = 0$ in $\Omega$, then for any $x_0 \in \mathbb{R}^n$, $0<r<R$, and $d > 0$,
\begin{equation*}
\begin{split}
&\int_{B_{r}(x_0)} \int_{B_{r}(x_0)} \frac{|(u_M^+(x)+d)^{\gamma/p} - (u_M^+(y)+d)^{\gamma/p}|^p}{|x-y|^{n+sp}} \,\mathrm{d}y \,\mathrm{d}x \\
&\leq C\left( 1+ \left( \frac{|\gamma|}{|\beta|} \right)^p \right) R^{n-sp} \left( \frac{R}{R-r} \right)^p \Phi_{\gamma;\, x_0, R}^{\gamma}(u_M^++d) \\
&\quad + C \frac{|\gamma|^p}{|\beta|} R^{n-sp} \left( \frac{R}{R-r} \right)^{n+sp} \Phi_{\beta;\, x_0, R}^{\beta}(u_M^++d) \, \mathrm{Tail}^{p-1}(u_M^++d; x_0, R),
\end{split}
\end{equation*}
where
\begin{equation} \label{eq:M}
M = \wsup_{B_R(x_0) \setminus \Omega} u_+, \quad u_M^+(x) = \max\lbrace u(x), M \rbrace,
\end{equation}
and $C = C(n, s, p, \Lambda) > 0$.
\end{lemma}

\begin{lemma} \label{lem:Caccio-supersoln}
Let $p \in (1, \infty)$ and assume $\gamma = \beta+p-1 < p-1$, $\gamma \neq 0$. If $u \in V^{s, p}(\Omega|\mathbb{R}^n)$ is a weak supersolution of $\mathcal{L}u = 0$ in $\Omega$ such that $u \geq 0$ in $B_{R}(x_0)$ for some ball $B_{R}(x_0) \subset \mathbb{R}^n$, then for any $0<r<R$ and $d > 0$,
\begin{equation*}
\begin{split}
&\int_{B_{r}(x_0)} \int_{B_{r}(x_0)} \frac{|(u_m^-(x)+d)^{\gamma/p} - (u_m^-(y)+d)^{\gamma/p}|^p}{|x-y|^{n+sp}} \,\mathrm{d}y \,\mathrm{d}x \\
&\leq C\left( 1+ \frac{|\gamma|^p}{|\beta|} + \left( \frac{|\gamma|}{|\beta|} \right)^p \right) R^{n-sp} \left( \frac{R}{R-r} \right)^p \Phi_{\gamma;\,x_0, R}^{\gamma}(u_m^-+d) \\
&\quad + C \frac{|\gamma|^p}{|\beta|} R^{n-sp} \left( \frac{R}{R-r} \right)^{n+sp} \Phi_{\beta;\, x_0, R}^{\beta}(u_m^-+d) \, \mathrm{Tail}^{p-1} ((u_m^-+d)_-; x_0, R),
\end{split}
\end{equation*}
where 
\begin{equation} \label{eq:m}
m=\winf_{B_{R}(x_0) \setminus \Omega}u, \quad u_m^-(x)=\min\{u(x), m\},
\end{equation}
and $C = C(n, s, p, \Lambda) > 0$.
\end{lemma}

\begin{lemma} \label{lem:Caccio-log}
Let $p \in (1, \infty)$. If $u \in V^{s, p}(\Omega|\mathbb{R}^n)$ is a weak supersolution of $\mathcal{L}u = 0$ in $\Omega$ such that $u \geq 0$ in $B_{R}(x_0)$ for some ball $B_{R}(x_0) \subset \mathbb{R}^n$, then for any $0<r<R$ and $d > 0$,
\begin{equation*}
\begin{split}
&\int_{B_{r}(x_0)} \int_{B_{r}(x_0)} \frac{|\log (u_m^-(x)+d) - \log (u_m^-(y)+d)|^p}{|x-y|^{n+sp}} \,\mathrm{d}y \,\mathrm{d}x \\
&\leq CR^{n-sp} \left( \left( \frac{R}{R-r} \right)^p + \left( \frac{R}{R-r} \right)^{n+sp} \Phi_{1-p;\, x_0, R}^{1-p}(u_m^-+d) \,\mathrm{Tail}^{p-1}((u_m^-+d)_-; x_0, R) \right),
\end{split}
\end{equation*}
where $u_m^-$ is defined as in \eqref{eq:m} and $C = C(n, s, p, \Lambda) > 0$.
\end{lemma}

\begin{proof} [Proof of \Cref{lem:Caccio-subsoln}, \Cref{lem:Caccio-supersoln}, and \Cref{lem:Caccio-log}]
Let us write $B_R = B_R(x_0)$ for simplicity. Let $\eta \in C_c^{\infty}(B_{(R+r)/2})$ be a cut-off function such that $\eta \equiv 1$ on $B_r$, $0 \leq \eta \leq 1$, and $|\nabla \eta| \leq C/(R-r)$. Note that $\beta > 0$ when $u$ is a weak subsolution and $\beta < 0$ when $u$ is a weak supersolution. For $d > 0$, we set $\bar{u} = u_M^++d$ if $u$ is a weak subsolution and $\bar{u} = u_m^-+d$ if $u$ is a weak supersolution. Note that, if $u$ is either a weak subsolution or supersolution, then so is $\bar{u}$. We define
\begin{equation*}
\varphi = 
\begin{cases}
(\bar{u}^{\beta} - \bar{M}^\beta) \eta^p &\text{if } \beta > 0, \\
(\bar{u}^{\beta} - \bar{m}^\beta) \eta^p &\text{if } \beta < 0,
\end{cases}
\end{equation*}
where $\bar{M} = M+d$ and $\bar{m} = m+d$. Then, $\varphi$ is a nonnegative test function in $W^{s, p}_0(\Omega)$. Thus, we have $\mathcal{E}^k(\bar{u}, \varphi) \leq 0$ if $u$ is a weak subsolution and $\mathcal{E}^k(\bar{u}, \varphi) \geq 0$ if $u$ is a weak supersolution. We split the double integral in $\mathcal{E}^k$ into two parts as
\begin{equation*}
\begin{split}
\mathcal{E}^k(\bar{u}, \varphi)
&= \int_{B_R} \int_{B_R} |\bar{u}(x)-\bar{u}(y)|^{p-2}(\bar{u}(x)-\bar{u}(y))(\varphi(x)-\varphi(y)) k(x, y) \,\mathrm{d}y \,\mathrm{d}x \\
&\quad+ 2\int_{B_R} \int_{\mathbb{R}^n \setminus B_R} |\bar{u}(x)-\bar{u}(y)|^{p-2}(\bar{u}(x)-\bar{u}(y)) \varphi(x) k(x, y) \,\mathrm{d}y \,\mathrm{d}x.
\end{split}
\end{equation*}
If $\beta \neq 0$ and $\gamma \neq 0$, then an application of \Cref{lem:alg-ineq1} with $a=\bar{u}(x)-d$, $b=\bar{u}(y)-d$, $\eta_1=\eta(x)$, $\eta_2=\eta(y)$, and $l = M$ if $\beta > 0$ and $l= m$ if $\beta < 0$ shows that
\begin{equation} \label{eq:Caccio-12}
\begin{split}
&\int_{B_R} \int_{B_R} \left| \bar{u}^{\gamma/p}(x)\eta(x)-\bar{u}^{\gamma/p}(y)\eta(y) \right|^{p} k(x, y) \,\mathrm{d}y \,\mathrm{d}x \\
&\leq C \left( 1+\left( \frac{|\gamma|}{|\beta|} \right)^p \right) \int_{B_R} \int_{B_R} \max\lbrace \bar{u}^{\gamma}(x), \bar{u}^{\gamma}(y) \rbrace |\eta(x)-\eta(y)|^p k(x, y) \,\mathrm{d}y \,\mathrm{d}x \\
&\quad - C \frac{|\gamma|^p}{\beta} \int_{B_R} \int_{\mathbb{R}^n \setminus B_R} |\bar{u}(x)-\bar{u}(y)|^{p-2} (\bar{u}(x)-\bar{u}(y)) \varphi(x) k(x, y) \,\mathrm{d}y \,\mathrm{d}x \\
&=: I_1 + I_2,
\end{split}
\end{equation}
where $C = C(p) > 0$. We obtain that, for $\beta > 0$,
\begin{equation} \label{eq:Caccio-2a}
\begin{split}
I_2
&\leq C \frac{|\gamma|^p}{|\beta|} \int_{B_R} \int_{\mathbb{R}^n \setminus B_R} (\bar{u}(y)-\bar{u}(x))^{p-1} (\bar{u}^\beta(x) - \bar{M}^\beta) \eta^p(x) {\bf 1}_{\lbrace \bar{u}(x) \leq \bar{u}(y) \rbrace} k(x, y) \,\mathrm{d}y \,\mathrm{d}x \\
&\leq C \frac{|\gamma|^p}{|\beta|} \int_{B_R} \int_{\mathbb{R}^n \setminus B_R} \bar{u}^{p-1}(y) \bar{u}^\beta(x) \eta^p(x) k(x, y) \,\mathrm{d}y \,\mathrm{d}x \\
&\leq C \frac{|\gamma|^p}{|\beta|} \int_{B_R} \bar{u}^{\beta}(x) \,\mathrm{d}x \left( \sup_{x \in \mathrm{supp} \, \eta} \int_{\mathbb{R}^n \setminus B_R} \bar{u}^{p-1}(y) k(x,y) \,\mathrm{d}y \right),
\end{split}
\end{equation}
and for $\beta < 0$, $\beta \neq -(p-1)$,
\begin{equation} \label{eq:Caccio-2b}
\begin{split}
I_2
&\leq C \frac{|\gamma|^p}{|\beta|} \int_{B_R} \int_{\mathbb{R}^n \setminus B_R} (\bar{u}(x)-\bar{u}(y))^{p-1} (\bar{u}^{\beta}(x) - \bar{m}^{\beta}) \eta^p(x) {\bf 1}_{\lbrace \bar{u}(x) \geq \bar{u}(y) \rbrace} k(x, y) \,\mathrm{d}y \,\mathrm{d}x \\
&\leq C \frac{|\gamma|^p}{|\beta|} \int_{B_R} \int_{\mathbb{R}^n \setminus B_R} (\bar{u}(x)+(\bar{u}(y))_-)^{p-1} \bar{u}^{\beta}(x) \eta^p(x) k(x, y) \,\mathrm{d}y \,\mathrm{d}x \\
&\leq C \frac{|\gamma|^p}{|\beta|} \int_{B_R} \int_{\mathbb{R}^n \setminus B_R} \bar{u}^{\gamma}(x) |\eta(x)-\eta(y)|^p k(x, y) \,\mathrm{d}y \,\mathrm{d}x \\
&\quad + C \frac{|\gamma|^p}{|\beta|} \int_{B_R} \bar{u}^{\beta}(x) \,\mathrm{d}x \left( \sup_{x \in \mathrm{supp} \,\eta} \int_{\mathbb{R}^n \setminus B_R} (\bar{u}(y))_-^{p-1} k(x,y) \,\mathrm{d}y \right),
\end{split}
\end{equation}
where $C = C(p) > 0$. If $x \in \mathrm{supp}\,\eta \subset B_{(R+r)/2}(x_0)$ and $y \in \mathbb{R}^n \setminus B_R(x_0)$, then $|x-y| \geq \frac{R-r}{2R}|y-x_0|$. Thus, we have, for $\beta > 0$,
\begin{equation} \label{eq:tail-sub}
\sup_{x \in \mathrm{supp}\,\eta} \int_{\mathbb{R}^n \setminus B_R} \bar{u}^{p-1}(y) k(x, y) \,\mathrm{d}y \leq \frac{\Lambda}{R^{sp}} \left( \frac{2R}{R-r} \right)^{n+sp} \mathrm{Tail}^{p-1}(\bar{u}; x_0, R),
\end{equation}
and for $\beta < 0$, $\beta \neq -(p-1)$,
\begin{equation} \label{eq:tail-super}
\sup_{x \in \mathrm{supp}\,\eta} \int_{\mathbb{R}^n \setminus B_R} (\bar{u}(y))_-^{p-1} k(x, y) \,\mathrm{d}y \leq \frac{\Lambda}{R^{sp}} \left( \frac{2R}{R-r} \right)^{n+sp} \mathrm{Tail}^{p-1}(\bar{u}_-; x_0, R).
\end{equation}
By combining \eqref{eq:Caccio-12}, \eqref{eq:Caccio-2a}, and \eqref{eq:tail-sub}, and using the estimate
\begin{equation} \label{eq:cut-off-est}
\begin{split}
&\int_{B_R} \bar{u}^{\gamma}(x) \int_{\mathbb{R}^n} |\eta(x)-\eta(y)|^p k(x, y) \,\mathrm{d}y \,\mathrm{d}x \\
&\leq\Lambda \int_{B_R} \bar{u}^{\gamma}(x) \left( \left( \frac{C}{R-r} \right)^p \int_{B_R(x)} \frac{|x-y|^{p}}{|x-y|^{n+sp}} \,\mathrm{d}y + \int_{\mathbb{R}^n \setminus B_R(x)} \frac{1}{|x-y|^{n+sp}} \,\mathrm{d}y \right) \,\mathrm{d}x \\
&\leq CR^{n-sp} \left( \frac{R}{R-r} \right)^p \Phi_{\gamma;\, x_0, R}^{\gamma}(\bar{u}),
\end{split}
\end{equation}
we conclude \Cref{lem:Caccio-subsoln}. Moreover, \eqref{eq:Caccio-12}, \eqref{eq:Caccio-2b}, \eqref{eq:tail-super}, and \eqref{eq:cut-off-est} prove \Cref{lem:Caccio-supersoln}.

If $\gamma = 0$, then \Cref{lem:alg-ineq-log} with $a=\bar{u}(x)-d$, $b=\bar{u}(y)-d$, $\eta_1=\eta(x)$, $\eta_2=\eta(y)$, and $l = m$ yields
\begin{equation*}
\begin{split}
&\int_{B_R} \int_{B_R} |\log \bar{u}(x) - \log \bar{u}(y)|^{p} \min \lbrace \eta(x), \eta(y) \rbrace^p k(x, y) \,\mathrm{d}y \,\mathrm{d}x \\
&\leq C \int_{B_R} \int_{B_R} |\eta(x)-\eta(y)|^p k(x, y) \,\mathrm{d}y \,\mathrm{d}x \\
&\quad + C \int_{B_R} \int_{\mathbb{R}^n \setminus B_R} |\bar{u}(x)-\bar{u}(y)|^{p-2} (\bar{u}(x)-\bar{u}(y)) \varphi(x) k(x, y) \,\mathrm{d}y \,\mathrm{d}x.
\end{split}
\end{equation*}
Then, similar arguments as above finish the proof of \Cref{lem:Caccio-log}.
\end{proof}


\subsection{Local boundedness up to the boundary} \label{sec:local-boundedness}


Having Caccioppoli-type estimates at hand, we prove the local boundedness and the weak Harnack inequality. We first study the local boundedness up to the boundary for weak subsolutions and then provide the boundary weak Harnack inequality for weak supersolutions. We refer the reader to \cite{BP16,DCKP16} for the interior estimates.

\begin{lemma} \label{lem:iteration-subsoln}
Let $p \in (1, n/s]$, $\gamma = \beta+p-1> p-1$, and
\begin{equation} \label{eq:chi}
\chi = 
\begin{cases}
n/(n-sp) &\text{if } p < n/s, \\
\text{any number larger than }1 &\text{if } p = n/s.
\end{cases}
\end{equation}
If $u \in V^{s, p}(\Omega|\mathbb{R}^n)$ is a bounded weak subsolution of $\mathcal{L}u=0$ in $\Omega$, then for any $x_0 \in \mathbb{R}^n$, $0<r<R$, and $d > 0$,
\begin{equation*}
\begin{split}
\Phi_{\gamma \chi;\, x_0, r}(u_M^++d)
&\leq (C (1+|\gamma|)^p)^{\frac{1}{\gamma}} \left( \frac{R}{r} \right)^{\frac{n}{\gamma}} \left( \frac{R}{R-r} \right)^{\frac{n+p}{\gamma}} \\
&\quad \times \left( 1 + \frac{\mathrm{Tail}^{p-1}(u_M^++d; x_0, R)}{d^{p-1}} \right)^{\frac{1}{\gamma}} \Phi_{\gamma;\, x_0, R}(u_M^++d),
\end{split}
\end{equation*}
where $\Phi$ and $u_M^+$ are given as in \eqref{eq:modular} and \eqref{eq:M}, respectively. The constant $C$ depends only on $n$, $s$, $p$, $\Lambda$, $\beta$ and $\chi$, and is bounded when $\beta$ is bounded away from zero.
\end{lemma}

\begin{proof}
By using \Cref{lem:Caccio-subsoln} and $(u_M^++d)^{\beta} \leq d^{1-p} (u_M^++d)^{\gamma}$, we have
\begin{equation*}
\begin{split}
&\int_{B_{r}(x_0)} \int_{B_{r}(x_0)} \frac{|(u_M^+(x)+d)^{\gamma/p} - (u_M^+(y)+d)^{\gamma/p}|^p}{|x-y|^{n+sp}} \,\mathrm{d}y \,\mathrm{d}x \\
&\leq C(1+|\gamma|)^p R^{n-sp} \left( \frac{R}{R-r} \right)^{n+p} \left( 1 + \frac{\mathrm{Tail}^{p-1}(u_M^++d; x_0, R)}{d^{p-1}} \right) \Phi_{\gamma;\, x_0, R}^{\gamma}(u_M^++d),
\end{split}
\end{equation*}
where $C = C(n, s, p, \Lambda, \beta) > 0$ is a constant, which is bounded when $\beta$ is bounded away from zero. If $p < n/s$, then by using the fractional Sobolev inequality, we obtain
\begin{equation*}
\begin{split}
&\Phi_{\gamma \chi;\, x_0, r}^{\gamma}(u_M^++d) \\
&= \frac{C}{r^{n-sp}} \|(u_M^++d)^{\gamma/p}\|_{L^{p^{\ast}}(B_r(x_0))}^{p} \\
&\leq \frac{C}{r^{n-sp}} \left( [(u_M^++d)^{\gamma/p}]_{W^{s, p}(B_r(x_0))}^p + r^{-sp} \|(u_M^++d)^{\gamma/p}\|_{L^p(B_r(x_0))}^p \right) \\
&\leq C (1+|\gamma|)^p \left( \frac{R}{r} \right)^{n} \left( \frac{R}{R-r} \right)^{n+p} \left( 1 + \frac{\mathrm{Tail}^{p-1}(u_M^++d; x_0, R)}{d^{p-1}} \right) \Phi_{\gamma;\, x_0, R}^{\gamma}(u_M^++d),
\end{split}
\end{equation*}
as desired. If $p = n/s$, let $q \in [1, p)$ and $\sigma \in (0,s)$ be such that $\chi = q^{\ast}_{\sigma} / p > 1$, where $q^{\ast}_{\sigma} = nq/(n-\sigma q)$. Then, the desired inequality can be proved as above by applying the fractional Sobolev inequality with $q < n/\sigma$, and then using \cite[Lemma 4.6]{Coz17} and H\"older's inequality.
\end{proof}

One can iterate \Cref{lem:iteration-subsoln} to obtain a local boundedness result for bounded weak subsolutions up to the boundary. Then, \Cref{thm:loc-bdd} follows by approximating a weak subsolution by a sequence of bounded weak subsolutions. Since the proof is the same as the one for interior estimates, we refer the reader to \cite[Theorem 3.8]{BP16} for the proof.

\begin{theorem} \label{thm:loc-bdd}
Let $p \in (1, n/s]$. If $u \in V^{s, p}(\Omega|\mathbb{R}^n)$ is a weak subsolution of $\mathcal{L}u=0$ in $\Omega$, then for any $x_0 \in \mathbb{R}^n$, $R>0$, and $\delta \in (0,1)$,
\begin{equation*}
\wsup_{B_{R/2}(x_0)} u_M^+ \leq \delta \, \mathrm{Tail}(u_M^+; x_0, R/2) + C \delta^{-\frac{(p-1)n}{sp^2}} \left( \fint_{B_{R}(x_0)} u_M^+(x)^p \,\mathrm{d}x \right)^{1/p},
\end{equation*}
where $u_M^+$ is given as in \eqref{eq:M} and $C = C(n, s, p, \Lambda) > 0$.
\end{theorem}

See \cite{KKP16,Pal18} for similar results on local boundedness up to the boundary, which are based on De Giorgi's iteration.


\subsection{Weak Harnack inequality up to the boundary} \label{sec:WHI}


In this section, we prove the boundary weak Harnack inequality for weak supersolutions. Let us first recall the interior estimates from \cite{DCKP14}.

\begin{theorem} \cite{DCKP14} \label{thm:interior-WHI}
Let $p \in (1, n/s)$ and $0 < t < (p-1)n/(n-sp)$. Let $u \in V^{s, p}(\Omega|\mathbb{R}^n)$ be a weak supersolution of $\mathcal{L}u=0$ in $\Omega$ such that $u \geq 0$ in $B_R(x_0) \subset \Omega$. Then, it holds that
\begin{equation} \label{eq:interior-WHI}
\left( \fint_{B_{R/2}(x_0)} u^t(x) \,\mathrm{d}x \right)^{1/t} \leq C \essinf_{B_{R/4}(x_0)} u + C \, \mathrm{Tail}(u_-; x_0, R),
\end{equation}
where $C > 0$ is a constant depending only on $n$, $s$, $p$, $t$, and $\Lambda$.
\end{theorem}

The authors of \cite{DCKP14} extended the techniques of \cite{DT84}, which use the Krylov--Safonov covering lemma, to nonlinear nonlocal operators to obtain the estimates \eqref{eq:interior-WHI} for small $t > 0$, and then used Moser's iteration with positive exponents to cover all $t < (p-1)n/(n-sp)$. The first part of their method can be replaced by Moser's iteration with negative exponents and John--Nirenberg-type inequality. Note that for interior estimates, Moser's iteration has been successfully applied to linear operators \cite{Kas09} and extended to nonlinear operators \cite{CK22}. The operators considered in \cite{CK22} have a different structure than the operator $\mathcal{L}$, but the underlying idea is the same.

The boundary weak Harnack inequality for weak supersolutions, the main result of this section, is stated as follows. We will prove \Cref{thm:bdry-WHI} by using Moser's iteration in the remaining part of this section. Note that we cover the case $p=n/s$, in which case we have $(p-1)n/(n-sp)=\infty$.

\begin{theorem} \label{thm:bdry-WHI}
Let $p \in (1, n/s]$ and $0 < t < (p-1)n/(n-sp)$. Let $u \in V^{s, p}(\Omega|\mathbb{R}^n)$ be a weak supersolution of $\mathcal{L}u = 0$ in $\Omega$ such that $u \geq 0$ in $B_{R}(x_0)$ for some ball $B_{R}(x_0) \subset \mathbb{R}^n$. Then, it holds that
\begin{equation*}
\left( \fint_{B_{R/2}(x_0)} u_m^-(x)^t \,\mathrm{d}x \right)^{1/t} \leq C \winf_{B_{R/4}(x_0)} u_m^- + C \, \mathrm{Tail}((u_m^-)_-; x_0, R),
\end{equation*}
where $u_m^-$ is defined as in \eqref{eq:m} and $C > 0$ is a constant depending only on $n$, $s$, $p$, $t$, and $\Lambda$.
\end{theorem}

We begin with the following lemma. We omit the proof of \Cref{lem:iteration-supersoln} since one can follow the proof of \Cref{lem:iteration-subsoln} and use \Cref{lem:Caccio-supersoln} instead of \Cref{lem:Caccio-subsoln}.

\begin{lemma} \label{lem:iteration-supersoln}
Let $p \in (1, n/s]$, $\gamma = \beta+p-1< p-1$, $\gamma \neq 0$, and $\chi$ be as in \eqref{eq:chi}. If $u \in V^{s, p}(\Omega|\mathbb{R}^n)$ is a weak supersolution of $\mathcal{L}u=0$ in $\Omega$ such that $u \geq 0$ in $B_R(x_0) \subset \mathbb{R}^n$, then for any and $0<r<R$ and $d > 0$,
\begin{equation*}
\begin{split}
\Phi_{\gamma \chi;\, x_0, r}(u_m^-+d)
&\leq (C(1+|\gamma|)^p)^{\frac{1}{|\gamma|}} \left( \frac{R}{r} \right)^{\frac{n}{|\gamma|}} \left( \frac{R}{R-r} \right)^{\frac{n+p}{|\gamma|}} \\
&\quad \times \left( 1 + \frac{\mathrm{Tail}^{p-1}((u_m^-+d)_-; x_0, R)}{d^{p-1}} \right)^{\frac{1}{|\gamma|}} \Phi_{\gamma;\, x_0, R}(u_m^-+d)
\end{split}
\end{equation*}
when $0<\gamma< p-1$, and
\begin{equation*}
\begin{split}
\Phi_{\gamma;\, x_0, R}(u_m^-+d)
&\leq (C(1+|\gamma|)^p)^{\frac{1}{|\gamma|}} \left( \frac{R}{r} \right)^{\frac{n}{|\gamma|}} \left( \frac{R}{R-r} \right)^{\frac{n+p}{|\gamma|}} \\
&\quad \times \left( 1 + \frac{\mathrm{Tail}^{p-1}((u_m^-+d)_-; x_0, R)}{d^{p-1}} \right)^{\frac{1}{|\gamma|}} \Phi_{\gamma \chi;\, x_0, r}(u_m^-+d)
\end{split}
\end{equation*}
when $\gamma <0$, where $u_m^-$ and $\Phi$ are given as in \eqref{eq:m} and \eqref{eq:modular}, respectively. The constant $C$ depends only on $n$, $s$, $p$, $\Lambda$, and $\beta$, and is bounded when $|\beta|$ is bounded away from zero.
\end{lemma}

The Caccoppoli-type estimate with $\gamma=0$, \Cref{lem:Caccio-log}, allows us to find an exponent $\bar{p}\in(0,1)$, which connects $\Phi_{-\bar{p}}$ and $\Phi_{\bar{p}}$.

\begin{lemma} \label{lem:iteration-log}
Let $p \in (1, \infty)$. Let $u \in V^{s, p}(\Omega|\mathbb{R}^n)$ be a weak supersolution of $\mathcal{L}u=0$ in $\Omega$ such that $u \geq 0$ in $B_R(x_0) \subset \mathbb{R}^n$. Assume that
\begin{equation} \label{eq:d}
d \geq \mathrm{Tail}((u_m^-)_-; x_0, R),
\end{equation}
then there exist $\bar{p} \in (0,1)$ and $C > 0$ such that
\begin{equation} \label{eq:p-bar}
\Phi_{\bar{p}; \, x_0, 3R/4}(u_m^-+d) \leq C\Phi_{-\bar{p}; \, x_0, 3R/4}(u_m^-+d),
\end{equation}
where $u_m^-$ is given as in \eqref{eq:m}. The constants $\bar{p}$ and $C$ depend only on $n$, $s$, $p$, and $\Lambda$.
\end{lemma}

\begin{proof}
Let $z_0 \in B_{3R/4}(x_0)$ and $r =R/16$ so that $B_{2r}(z_0) \subset B_{7R/8}(x_0)$. By applying \Cref{lem:Caccio-log} to $u$ in $B_{r}(z_0) \subset B_{2r}(z_0)$, we have
\begin{equation} \label{eq:log1}
\begin{split}
&\int_{B_{r}(z_0)} \int_{B_{r}(z_0)} \frac{|\log v(x) - \log v(y)|^p}{|x-y|^{n+sp}} \,\mathrm{d}y \,\mathrm{d}x \\
&\leq Cr^{n-sp} \left( 1 + \Phi_{1-p;\, z_0, R}^{1-p}(v) \,\mathrm{Tail}^{p-1}(v_-; z_0, 2r) \right) \\
&\leq Cr^{n-sp} \left( 1 + \frac{\mathrm{Tail}^{p-1}(v_-; z_0, 2r)}{d^{p-1}} \right),
\end{split}
\end{equation}
where we denote $v = u_m^-+d$. Since $v > 0$ in $B_R(x_0)$ and $v_- \leq (u_m^-)_-$, we obtain
\begin{equation} \label{eq:log2}
\frac{\mathrm{Tail}^{p-1}(v_-; z_0, 2r)}{d^{p-1}} \leq C \frac{\mathrm{Tail}^{p-1}((u_m^-)_-; x_0, R)}{d^{p-1}} \leq C
\end{equation}
by the assumption \eqref{eq:d}.

By the fractional Poincar\'e inequality, \Cref{poincare1}, we obtain
\begin{equation} \label{eq:log3}
\left\|\log v-(\log v)_{B_r(z_0)} \right\|_{L^p(B_r(z_0))}^p \leq Cr^{sp}[\log v]_{W^{s, p}(B_r(z_0))}^p.
\end{equation}
Thus, it follows from \eqref{eq:log1}, \eqref{eq:log2}, \eqref{eq:log3}, and H\"older's inequality that
\begin{equation*}
\int_{B_r(z_0)} |\log v(x)-(\log v)_{B_r(z_0)}| \,\mathrm{d}x \leq Cr^{n(1-1/p)} \left\|\log v-(\log v)_{B_r(z_0)} \right\|_{L^p(B_r(z_0))} \leq Cr^{n},
\end{equation*}
which implies that $\log v$ is a function of bounded mean oscillation. By the John--Nirenberg embedding, there exists $\bar{p} \in (0,1)$ such that
\begin{equation*}
\fint_{B_{3R/4}(x_0)} \exp \left( \bar{p} |\log v(x) - (\log v)_{B_{3R/4}(x_0)}| \right) \mathrm{d}x \leq C.
\end{equation*}
Therefore, we conclude that
\begin{equation*}
\Phi_{\bar{p};\,x_0, 3R/4}^{\bar{p}}(v) \Phi_{-\bar{p};\,x_0, 3R/4}^{-\bar{p}}(v) \leq C \exp\left( \bar{p} (\log v)_{B_{3R/4}(x_0)} \right)\exp\left( -\bar{p} (\log v)_{B_{3R/4}(x_0)} \right) \leq C
\end{equation*}
for some $C = C(n, s, p, \Lambda) > 0$.
\end{proof}

We now prove \Cref{thm:bdry-WHI} by using Moser's iteration technique. Since the proof is quite standard, let us provide a sketch of proof.

\begin{proof} [Proof of \Cref{thm:bdry-WHI}]
Let $\bar{p}$ be the exponent given in \Cref{lem:iteration-log}. By setting
\begin{equation} \label{eq:d-tail}
d = \mathrm{Tail}((u_m^-)_-; x_0, R),
\end{equation}
we have \eqref{eq:p-bar}. Moreover, we have from \Cref{lem:iteration-supersoln} that for any $1/4\leq \rho<\tau \leq 1$,
\begin{equation} \label{eq:Phi-positive}
\Phi_{\gamma \chi;\, x_0, \rho R}(u_m^-+d) \leq \left( \frac{C(1+|\gamma|)}{\tau - \rho} \right)^{\frac{n+p}{|\gamma|}} \Phi_{\gamma;\, x_0, \tau R}(u_m^-+d)
\end{equation}
when $0<\gamma< p-1$, and
\begin{equation} \label{eq:Phi-negative}
\Phi_{\gamma;\, x_0, \tau R}(u_m^-+d) \leq \left( \frac{C(1+|\gamma|)}{\tau-\rho} \right)^{\frac{n+p}{|\gamma|}} \Phi_{\gamma \chi;\, x_0, \rho R}(u_m^-+d)
\end{equation}
when $\gamma <0$. The constant $C$ is bounded when $|\beta|$ is bounded away from zero.

By iterating \eqref{eq:Phi-negative} for $j=0,1,\dots$, with
\begin{equation*}
\rho = \frac{1+2^{-j+1}}{4}, \quad \tau = \frac{1+2^{-j+2}}{4}, \quad\text{and}\quad \gamma = -\bar{p} \chi^{j},
\end{equation*}
we obtain
\begin{equation} \label{eq:p-bar-inf}
\Phi_{-\bar{p}; \, x_0, 3R/4}(u_m^-+d) \leq C \winf_{B_{R/4}(x_0)} (u_m^-+d).
\end{equation}
If $0< t\leq \bar{p}$, it readily follows from \eqref{eq:p-bar}, \eqref{eq:p-bar-inf}, and H\"older's inequality that
\begin{equation} \label{eq:WHI-d}
\Phi_{t; \, x_0, R/2}(u_m^-+d) \leq C \winf_{B_{R/4}(x_0)} (u_m^-+d).
\end{equation}
Otherwise, we iterate \eqref{eq:Phi-positive} in a similar manner so that
\begin{equation} \label{eq:p-bar-Lt}
\Phi_{t; \, x_0, R/2}(u_m^-+d) \leq C\Phi_{\bar{p}; \, x_0, 3R/4}(u_m^-+d).
\end{equation}
In this case, \eqref{eq:WHI-d} follows from \eqref{eq:p-bar-Lt}, \eqref{eq:p-bar}, and \eqref{eq:p-bar-inf}.

Finally, using the triangle inequality, \eqref{eq:WHI-d}, and \eqref{eq:d-tail}, we conclude the theorem.
\end{proof}


\section{The sufficient condition for the regularity of a boundary point} \label{sec:sufficient}


We are ready to prove the sufficient condition for a regular boundary point with respect to $\mathcal{L}$. As mentioned in \Cref{sec:introduction}, we focus on the case $p \in (1, n/s]$. We mainly follow the ideas of \cite{GZ77, HV21}, but further analysis is required to treat the tail term.

Let us fix a boundary point $x_0 \in \partial \Omega$. Let $g \in V^{s, p}(\Omega|\mathbb{R}^n) \cap C(\mathbb{R}^n)$ and let $u\in V^{s, p}(\Omega|\mathbb{R}^n) \cap C(\Omega)$ be the $\mathcal{L}$-harmonic function with $u-g \in W^{s, p}_0(\Omega)$. We will prove \eqref{eq:limit}, or equivalently,
\begin{equation} \label{eq:contrad}
\lim_{\rho \to 0} \sup_{\Omega \cap B_{\rho}(x_0)}u \leq g(x_0) \quad \text{and} \quad \lim_{\rho \to 0} \inf_{\Omega \cap B_{\rho}(x_0)}u \geq g(x_0),
\end{equation}
provided that the Wiener integral diverges. By symmetry, we only prove the first inequality of \eqref{eq:contrad}. We assume to the contrary that
\begin{equation} \label{eq:L}
L:=\lim_{\rho \to 0} \sup_{\Omega \cap B_{\rho}(x_0)}u > g(x_0)
\end{equation}
and choose $l \in \mathbb{R}$ so that $g(x_0)<l<L$. By the continuity of $g$, there exists a sufficiently small radius $r_{\ast}>0$ such that
\begin{equation*}
l \geq \sup_{\overline{B_r(x_0)} \setminus \Omega}g \quad \text{for any } r \in (0, r_{\ast}).
\end{equation*}
We set
\begin{equation*}
M_l(r)=\wsup_{B_{r}(x_0)} \, (u-l)_+ \quad \text{and}\quad u_{l, r}=M_l(r)-(u-l)_+ \quad \text{for } r>0.
\end{equation*}
Note that since $\lim_{r \to 0} M_l(r)=L-l>0$, we have $M_{l}(r) \geq L-l>0$ for any $r >0$. Moreover, $u_{l, r}$ is a weak supersolution of $\mathcal{L}u=0$ in $\Omega$. Recall the notation $u_m^-$ in \eqref{eq:m}. Since $(u-l)_+=0$ in $B_r(x_0) \setminus \Omega$ in the sense of $V^{s, p}(\Omega|\mathbb{R}^n)$, we have
\begin{equation*}
\winf_{B_{r}(x_0) \setminus \Omega} u_{l, r} = \wsup_{B_{r}(x_0) \setminus \Omega} u_{l, r} = M_l(r),
\end{equation*}
and hence $(u_{l, r})_m^- = u_{l, r}$. 

The following lemma provides a key estimate of the $(s, p)$-capacity of a compact set
\begin{equation} \label{eq:D}
D_{\rho}(x_0) := \overline{B_{\rho}(x_0)} \setminus \Omega
\end{equation}
with respect to $B_{2\rho}(x_0)$, which will be used to find a contradiction in the end of this section.

\begin{lemma} \label{lem:key}
There exists a constant $C = C(n, s, p, \Lambda) > 0$ such that
\begin{equation*}
\mathrm{cap}_{s,p}(D_{\rho}(x_0), B_{2\rho}(x_0)) \leq C\rho^{n-sp} \left(\frac{M_l(4\rho)-M_l(\rho)+\mathrm{Tail}((u_{l, 4\rho})_-; x_0, 4\rho)}{M_l(4\rho)} \right)^{p-1}
\end{equation*}
for any $\rho \in (0, r_{\ast}/4)$, where $D_{\rho}(x_0)$ is given by \eqref{eq:D}.
\end{lemma}

\begin{proof}
Let us write $B_{\rho} = B_{\rho}(x_0)$ and $D_{\rho} = D_{\rho}(x_0)$ for simplicity. To begin with, we define
\begin{equation*}
v_{l, 4\rho}:=\frac{\eta \, u_{l, 4\rho}}{M_l(4\rho)},
\end{equation*}
where $\eta \in C_c^{\infty}(B_{3\rho/2})$ is a cut-off function such that $\eta \equiv 1$ on $\overline{B_{\rho}}$, $0 \leq \eta \leq 1$, and $|\nabla \eta| \leq C/\rho$. It is immediately checked that $v_{l, 4\rho}$ is admissible for $\mathrm{cap}_{s,p}(D_{\rho}, B_{2\rho})$, and hence
\begin{equation*}
\begin{split}
\mathrm{cap}_{s,p}(D_{\rho}, B_{2\rho})
&\leq \int_{\mathbb{R}^n} \int_{\mathbb{R}^n} \frac{|v_{l, 4\rho}(x) - v_{l, 4\rho}(y)|^p}{|x-y|^{n+sp}} \,\mathrm{d}y \,\mathrm{d}x \\
&\leq 2\int_{B_{3\rho/2}} \int_{\mathbb{R}^n} \frac{|v_{l, 4\rho}(x) - v_{l, 4\rho}(y)|^p}{|x-y|^{n+sp}} \,\mathrm{d}y \,\mathrm{d}x =: I.
\end{split}
\end{equation*}
Since
\begin{equation*}
|v_{l, 4\rho}(x)-v_{l, 4\rho}(y)|^p \leq \frac{2^{p-1}}{M_l^p(4\rho)} \left( u_{l, 4\rho}^p(x) |\eta(x)-\eta(y)|^p + |u_{l, 4\rho}(x)-u_{l, 4\rho}(y)|^p \eta^p(y) \right)
\end{equation*}
and
\begin{equation} \label{eq:carre-du-champ}
\int_{\mathbb{R}^n} \frac{|\eta(x)-\eta(y)|^p}{|x-y|^{n+sp}} \,\mathrm{d}y \leq \frac{C}{\rho^p} \int_{\mathbb{R}^n} \frac{|x-y|^p \land \rho^p}{|x-y|^{n+sp}} \,\mathrm{d}y \leq C \rho^{-sp},
\end{equation}
we have
\begin{equation} \label{eq:cap-v}
\begin{split}
I
&\leq C \frac{\rho^{n-sp}}{M_l^p(4\rho)} \fint_{B_{3\rho/2}} u_{l, 4\rho}^p(x) \,\mathrm{d}x + \frac{C}{M_l^p(4\rho)} \int_{B_{3\rho/2}} \int_{\mathbb{R}^n} \frac{|u_{l, 4\rho}(x)-u_{l, 4\rho}(y)|^p}{|x-y|^{n+sp}} \eta^p(y) \,\mathrm{d}y \,\mathrm{d}x \\
&\leq C \frac{\rho^{n-sp}}{M_l^{p-1}(4\rho)} \fint_{B_{2\rho}} u_{l, 4\rho}^{p-1}(x) \,\mathrm{d}x + \frac{C}{M_l^p(4\rho)} \int_{B_{3\rho/2}} \int_{B_{3\rho/2}} \frac{|u_{l, 4\rho}(x)-u_{l, 4\rho}(y)|^p}{|x-y|^{n+sp}} \,\mathrm{d}y \,\mathrm{d}x.
\end{split}
\end{equation}
For the first term on the right-hand side of \eqref{eq:cap-v}, since $u_{l, 4\rho}$ is a weak supersolution such that $u_{l, 4\rho} \geq 0$ in $B_{4\rho}$, we apply \Cref{thm:bdry-WHI} to obtain
\begin{equation*}
\begin{split}
\frac{\rho^{n-sp}}{M_l^{p-1}(4\rho)} \fint_{B_{2\rho}} u_{l, 4\rho}^{p-1}(x) \,\mathrm{d}x
&\leq C \frac{\rho^{n-sp}}{M_l^{p-1}(4\rho)} \left( \winf_{B_{\rho}} u_{l, 4\rho} +\mathrm{Tail}((u_{l, 4\rho})_-; x_0, 4\rho) \right)^{p-1}\\
&= C \rho^{n-sp} \left(\frac{M_l(4\rho)-M_l(\rho)+\mathrm{Tail}((u_{l, 4\rho})_-; x_0, 4\rho)}{M_l(4\rho)} \right)^{p-1}.
\end{split}
\end{equation*}
It only remains to estimate the last term on the right-hand side of \eqref{eq:cap-v}, but the desired estimate follows from \Cref{lem:gradient-est} below.
\end{proof}

\begin{lemma} \label{lem:gradient-est}
There exists a constant $C = C(n, s, p, \Lambda) > 0$ such that
\begin{equation*}
\begin{split}
&\frac{1}{M_l^p(4\rho)} \int_{B_{3\rho/2}(x_0)} \int_{B_{3\rho/2}(x_0)} \frac{|u_{l, 4\rho}(x)-u_{l, 4\rho}(y)|^p}{|x-y|^{n+sp}} \,\mathrm{d}y \,\mathrm{d}x \\
&\leq C\rho^{n-sp} \left(\frac{M_l(4\rho)-M_l(\rho)+\mathrm{Tail}((u_{l, 4\rho})_-; x_0, 4\rho)}{M_l(4\rho)} \right)^{p-1}
\end{split}
\end{equation*}
for any $\rho \in (0, r_{\ast}/4)$.
\end{lemma}

\begin{proof}
We write $B_r = B_r(x_0)$ for simplicity. Let $\eta \in C_c^{\infty}(B_{13\rho/8})$ be a cut-off function satisfying $\eta \equiv 1$ on $\overline{B_{3\rho/2}}$, $0 \leq \eta \leq 1$, and $|\nabla \eta| \leq C/\rho$. Then, $\varphi = (u-l)_+ \eta^p \in W^{s, p}_0(\Omega)$. Since $u_{l, 4\rho}$ is a weak supersolution of $\mathcal{L}u=0$ in $\Omega$, we obtain
\begin{equation*}
\begin{split}
0
&\leq \int_{B_{7\rho/4}}\int_{B_{7\rho/4}} |u_{l, 4\rho}(x)-u_{l, 4\rho}(y)|^{p-2}(u_{l, 4\rho}(x)-u_{l, 4\rho}(y)) (\varphi(x) - \varphi(y)) k(x, y) \,\mathrm{d}y\,\mathrm{d}x \\
&\quad + 2\int_{B_{7\rho/4}}\int_{\mathbb{R}^n \setminus B_{7\rho/4}} |u_{l, 4\rho}(x)-u_{l, 4\rho}(y)|^{p-2}(u_{l, 4\rho}(x)-u_{l, 4\rho}(y)) \varphi(x) k(x,y) \,\mathrm{d}y\,\mathrm{d}x.
\end{split}
\end{equation*}
By using $\varphi(x) - \varphi(y) = (u(x)-l)_+ (\eta^p(x) - \eta^p(y)) - (u_{l, 4\rho}(x) - u_{l, 4\rho}(y)) \eta^p(y)$ and the ellipticity assumption \eqref{eq:ellipticity}, we have
\begin{equation} \label{eq:suf-12}
\begin{split}
&\int_{B_{3\rho/2}} \int_{B_{3\rho/2}} \frac{|u_{l, 4\rho}(x)-u_{l, 4\rho}(y)|^p}{|x-y|^{n+sp}} \,\mathrm{d}y \,\mathrm{d}x \\
&\leq \Lambda \int_{B_{7\rho/4}} \int_{B_{7\rho/4}} |u_{l, 4\rho}(x)-u_{l, 4\rho}(y)|^p \eta^p(y) k(x, y) \,\mathrm{d}y \,\mathrm{d}x \\
&\leq \Lambda \int_{B_{7\rho/4}} \int_{B_{7\rho/4}} |u_{l, 4\rho}(x)-u_{l, 4\rho}(y)|^{p-1} (u(x)-l)_+ |\eta^p(x)-\eta^p(y)| k(x, y) \,\mathrm{d}y \,\mathrm{d}x \\
&\quad + 2\Lambda \int_{B_{7\rho/4}}\int_{\mathbb{R}^n \setminus B_{7\rho/4}} |u_{l, 4\rho}(x)-u_{l, 4\rho}(y)|^{p-2}(u_{l, 4\rho}(x)-u_{l, 4\rho}(y)) \varphi(x) k(x,y) \,\mathrm{d}y\,\mathrm{d}x \\
&=: I_1 + I_2.
\end{split}
\end{equation}

Let us first estimate $I_1$. Since $(u(x)-l)_+ \leq M_l(4\rho)$ and $|\eta^p(x)-\eta^p(y)| \leq C|x-y|/\rho$, we estimate $I_1$ as
\begin{equation*}
I_1 \leq C M_l(4\rho) \int_{B_{7\rho/4}} \int_{B_{7\rho/4}} \frac{|v(x)-v(y)|^{p-1}}{|x-y|^{n+sp}} \frac{|x-y|}{\rho} \,\mathrm{d}y \,\mathrm{d}x,
\end{equation*}
where $v := u_{l, 4\rho} + d$ for some $d > 0$, which will be chosen later. Let $\gamma > 0$ be sufficiently close to $p-1$ so that $1 < p-\gamma < n/(n-sp)$. By H\"older's inequality and \Cref{lem:alg-ineq-suff}, we have
\begin{equation*}
\begin{split}
I_1
&\leq CM_l(4\rho) \left( \int_{B_{7\rho/4}} \int_{B_{7\rho/4}} \max\lbrace v(x), v(y) \rbrace^{(p-\gamma)(p-1)} \left( \frac{|x-y|}{\rho} \right)^p \frac{\mathrm{d}y \,\mathrm{d}x}{|x-y|^{n+sp}} \right)^{\frac{1}{p}} \\
&\quad \times \left( \int_{B_{7\rho/4}} \int_{B_{7\rho/4}} \max\lbrace v(x), v(y) \rbrace^{\gamma-p} \frac{|v(x)-v(y)|^p}{|x-y|^{n+sp}} \,\mathrm{d}y \,\mathrm{d}x \right)^{\frac{p-1}{p}} \\
&\leq C\frac{M_l(4\rho)}{\rho^s} \left( \int_{B_{2\rho}} v^{(p-\gamma)(p-1)}(x) \,\mathrm{d}x \right)^{\frac{1}{p}} \left( \int_{B_{7\rho/4}} \int_{B_{7\rho/4}} \frac{|v^{\gamma/p}(x)-v^{\gamma/p}(y)|^p}{|x-y|^{n+sp}} \,\mathrm{d}y \,\mathrm{d}x \right)^{\frac{p-1}{p}}.
\end{split}
\end{equation*}
Since $0<(p-\gamma)(p-1) < (p-1)n/(n-sp)$ and $v_m^- = v$, \Cref{thm:bdry-WHI} applied to a weak supersolution $v$ shows that
\begin{equation*}
\int_{B_{2\rho}} v^{(p-\gamma)(p-1)}(x) \,\mathrm{d}x \leq C \rho^n \left( \winf_{B_{\rho}} v + \mathrm{Tail}(v_-; x_0, 4\rho) \right)^{(p-\gamma)(p-1)}.
\end{equation*}
Moreover, since $\gamma < p-1$, it follows from \Cref{lem:Caccio-supersoln} that
\begin{equation*}
\begin{split}
&\int_{B_{7\rho/4}} \int_{B_{7\rho/4}} \frac{|v^{\gamma/p}(x)-v^{\gamma/p}(y)|^p}{|x-y|^{n+sp}} \,\mathrm{d}y \,\mathrm{d}x \\
&\leq C\rho^{n-sp} \left( \fint_{B_{2\rho}} v^{\gamma}(x) \,\mathrm{d}x + \left( \fint_{B_{2\rho}} v^{\beta}(x) \,\mathrm{d}x \right) \, \mathrm{Tail}^{p-1} (v_-; x_0, 2\rho) \right),
\end{split}
\end{equation*}
where $\beta = \gamma-(p-1)$. We set $d = \mathrm{Tail}((u_{l, 4\rho})_-; x_0, 4\rho)$, then since
\begin{equation*}
\mathrm{Tail}(v_-; x_0, 2\rho) \leq \mathrm{Tail}((u_{l, 4\rho})_-; x_0, 2\rho) \leq d,
\end{equation*}
we have
\begin{equation*}
\begin{split}
\left( \fint_{B_{2\rho}} v^{\beta}(x) \,\mathrm{d}x \right) \, \mathrm{Tail}^{p-1} (v_-; x_0, 2\rho)
&\leq \left( \fint_{B_{2\rho}} v^{\gamma}(x) \,\mathrm{d}x \right) \, \frac{\mathrm{Tail}^{p-1} (v_-; x_0, 2\rho)}{d^{p-1}} \leq \fint_{B_{2\rho}} v^{\gamma}(x) \,\mathrm{d}x.
\end{split}
\end{equation*}
Applying \Cref{thm:bdry-WHI} yields
\begin{equation*}
\fint_{B_{2\rho}} v^{\gamma}(x) \,\mathrm{d}x \leq C \left( \winf_{B_{\rho}} v + \mathrm{Tail}(v_-; x_0, 4\rho) \right)^{\gamma}.
\end{equation*}
Therefore, we obtain
\begin{equation} \label{eq:suf-1}
\begin{split}
I_1
&\leq C M_l(4\rho) \rho^{n-sp} \left( \winf_{B_{\rho}} v + \mathrm{Tail}(v_-; x_0, 4\rho) \right)^{p-1} \\
&\leq C M_l(4\rho) \rho^{n-sp} (M_l(4\rho)-M_l(\rho)+\mathrm{Tail}((u_{l, 4\rho})_-; x_0, 4\rho))^{p-1},
\end{split}
\end{equation}
where $C = C(n, s, p, \Lambda) > 0$.

Let us next estimate $I_2$. Since $\varphi(x) \leq M_l(4\rho) \eta^p(x)$ and
\begin{equation*}
|u_{l, 4\rho}(x)-u_{l, 4\rho}(y)|^{p-2}(u_{l, 4\rho}(x)-u_{l, 4\rho}(y)) \leq C \left( u_{l, 4\rho}^{p-1}(x) + (u_{l, 4\rho})_-^{p-1}(y) \right),
\end{equation*}
we have
\begin{equation*}
\begin{split}
I_2
&\leq CM_l(4\rho) \int_{B_{7\rho/4}}\int_{\mathbb{R}^n \setminus B_{7\rho/4}} \left( u_{l, 4\rho}^{p-1}(x) + (u_{l, 4\rho})_-^{p-1}(y) \right) \eta^p(x) k(x,y) \,\mathrm{d}y\,\mathrm{d}x \\
&\leq CM_l(4\rho) \int_{B_{7\rho/4}} u_{l, 4\rho}^{p-1}(x) \int_{\mathbb{R}^n \setminus B_{7\rho/4}} \frac{|\eta(x)-\eta(y)|^p}{|x-y|^{n+sp}} \,\mathrm{d}y\,\mathrm{d}x \\
&\quad + CM_l(4\rho) \int_{B_{7\rho/4}} \eta^{p}(x) \int_{\mathbb{R}^n \setminus B_{7\rho/4}} \frac{(u_{l, 4\rho})_-^{p-1}(y)}{|x-y|^{n+sp}} \,\mathrm{d}y\,\mathrm{d}x.
\end{split}
\end{equation*}
By using similar arguments as in \eqref{eq:carre-du-champ} and \eqref{eq:tail-super}, and observing that $u_{l, 4\rho} \geq 0$ in $B_{4\rho}$, we arrive at
\begin{equation*}
I_2 \leq C M_l(4\rho) \rho^{n-sp} \left( \fint_{B_{2\rho}} u_{l, 4\rho}^{p-1}(x) \,\mathrm{d}x + \mathrm{Tail}^{p-1}((u_{l, 4\rho})_-; x_0, 4\rho) \right).
\end{equation*}
By applying \Cref{thm:bdry-WHI}, we conclude that
\begin{equation} \label{eq:suf-2}
I_2 \leq C M_l(4\rho) \rho^{n-sp} (M_l(4\rho)-M_l(\rho)+\mathrm{Tail}((u_{l, 4\rho})_-; x_0, 4\rho))^{p-1}.
\end{equation}
The desired result follows by combining \eqref{eq:suf-12}, \eqref{eq:suf-1}, and \eqref{eq:suf-2}.
\end{proof}

We are now ready to prove the sufficiency of the Wiener criterion.

\begin{proof}[Proof of the sufficient part of \Cref{thm:main}]
As explained in the beginning of this section, we assume to the contrary that \eqref{eq:L} holds and find a contradiction. By \Cref{lem:key}, we have
\begin{equation*}
\frac{\mathrm{cap}_{s, p}(D_{\rho}(x_0), B_{2\rho}(x_0))}{\rho^{n-sp}} \leq C \left( \frac{M_l(4\rho) - M_l(\rho) + \mathrm{Tail}((u_{l, 4\rho})_-; x_0, 4\rho)}{M_l(4\rho)} \right)^{p-1}
\end{equation*}
for any $\rho \in (0, r_{\ast}/4)$. Using this inequality and $M_l(4\rho) \geq L-l$, we derive an estimate for the Wiener integral:
\begin{equation} \label{eq:Wiener-I12}
\begin{split}
&\int_0^{r_{\ast}/4} \left( \frac{\mathrm{cap}_{s,p}(D_{\rho}(x_0), B_{2\rho}(x_0))}{\rho^{n-sp}} \right)^{\frac{1}{p-1}} \frac{\mathrm{d}\rho}{\rho} \\
&\leq C \int_0^{r_{\ast}/4} \frac{M_l(4\rho)-M_l(\rho)+\mathrm{Tail}((u_{l, 4\rho})_-; x_0, 4\rho)}{\rho \, M_l(4\rho)} \,\mathrm{d}\rho\\
&\leq C \left( \int_0^{r_{\ast}/4}\frac{M_l(4\rho)-M_l(\rho)}{\rho}\,\mathrm{d}\rho + \int_0^{r_{\ast}/4}\frac{\mathrm{Tail}((u_{l, 4\rho})_-; x_0, 4\rho)}{\rho}\,\mathrm{d}\rho \right) =: C(I_1 + I_2).
\end{split}
\end{equation}
By a change of variables and \Cref{thm:loc-bdd}, we have
\begin{equation} \label{eq:Wiener-I1}
I_1 = \int_{r_{\ast}/4}^{r_{\ast}} \frac{M_l(\rho)}{\rho}\,\mathrm{d}\rho \leq (\log 4) \left( \wsup_{B_{r_{\ast}}(x_0)} u + |l| \right) < +\infty.
\end{equation}
Thus, the rest part of the proof is devoted to the estimate of $I_2$.

Let $N = N(\rho) \in \mathbb{N}$ be the unique integer satisfying $4^{-N-1}r_{\ast} \leq \rho < 4^{-N}r_{\ast}$. We decompose the tail term by
\begin{equation*}
\mathrm{Tail}^{p-1}((u_{l, 4\rho})_-; x_0, 4\rho) \leq (4\rho)^{sp} \left( \int_{\mathbb{R}^n \setminus B_{r_{\ast}}(x_0)} \frac{(u_{l, 4\rho})_-^{p-1}}{|y-x_0|^{n+sp}} \,\mathrm{d}y + \sum_{j=1}^{N} \int_{A_j(x_0)} \frac{(u_{l, 4\rho})_-^{p-1}}{|y-x_0|^{n+sp}} \,\mathrm{d}y \right),
\end{equation*}
where $A_j(x_0)$ denotes the annulus $B_{4^{j+1}\rho}(x_0) \setminus B_{4^j\rho}(x_0)$. First of all, since $(u_{l, 4\rho})_- \leq |u| + |l|$, we have
\begin{equation} \label{eq:estimate1}
\rho^{sp} \int_{\mathbb{R}^n \setminus B_{r_{\ast}}(x_0)} \frac{(u_{l, 4\rho})_-^{p-1}}{|y-x_0|^{n+sp}} \,\mathrm{d}y \leq C \left( \frac{\rho}{r_{\ast}} \right)^{sp} \left( \mathrm{Tail}^{p-1}(u; x_0, r_{\ast}) + |l|^{p-1} \right).
\end{equation}
We next observe that $(u_{l, 4\rho})_- \leq M_l(4^{j+1}\rho) - M_l(4\rho)$ in $A_j(x_0)$. Thus, we obtain
\begin{equation} \label{eq:estimate2}
\rho^{sp} \int_{A_j(x_0)} \frac{(u_{l, 4\rho})_-^{p-1}}{|y-x_0|^{n+sp}} \,\mathrm{d}y \leq C 4^{-spj} \left( M_l(4^{j+1}\rho) - M_l(4\rho) \right)^{p-1}.
\end{equation}
By combining two estimates \eqref{eq:estimate1} and \eqref{eq:estimate2}, we have
\begin{equation*}
\begin{split}
\mathrm{Tail}^{p-1}((u_{l, 4\rho})_-; x_0, 4\rho)
&\leq C \left( \frac{\rho}{r_{\ast}} \right)^{sp} \left( \mathrm{Tail}^{p-1}(u; x_0, r_{\ast}) + |l|^{p-1} \right) \\
&\quad + C \sum_{j=1}^{N} 4^{-spj} \left( M_l(4^{j+1}\rho) - M_l(4\rho) \right)^{p-1}.
\end{split}
\end{equation*}
To proceed, we use \cite[Lemma 3 (b)]{Dyd06}: for every $t > 0$ and $\beta > 1$ there exists a constant $C > 0$ such that for any sequence $\lbrace a_j \rbrace$ of nonnegative numbers, we have
\begin{equation} \label{eq:Dyd06}
\left( \sum_{j=1}^{\infty} a_j \right)^{t} \leq C \sum_{j=1}^{\infty} \beta^j a_j^t.
\end{equation}
An application of \eqref{eq:Dyd06} with $t = 1/(p-1)$ and $\beta = 2^{\frac{sp}{p-1}} > 1$ yields
\begin{equation*}
\mathrm{Tail}((u_{l, 4\rho})_-; x_0, 4\rho) \leq C \left( \frac{\rho}{r_{\ast}} \right)^{\frac{sp}{p-1}} (\mathrm{Tail}(u; x_0, r_{\ast}) + |l|) + C \sum_{j=1}^{N} \frac{M_l(4^{j+1}\rho) - M_l(4\rho)}{2^{\frac{sp}{p-1}j}},
\end{equation*}
and hence
\begin{equation*}
I_2 \leq C(\mathrm{Tail}(u; x_0, r_{\ast}) + |l|) + C \int_0^{r_{\ast}/4} \sum_{j=1}^{\infty} \left( 2^{-\frac{sp}{p-1}j} \frac{M_l(4^{j+1}\rho) - M_l(4\rho)}{\rho} {\bf 1}_{\lbrace 4^j \rho < r_{\ast} \rbrace} \right) \mathrm{d}\rho.
\end{equation*}
By using Fubini's theorem and a change of variables, we have
\begin{equation*}
\begin{split}
&\int_0^{r_{\ast}/4} \sum_{j=1}^{\infty} \left( 2^{-\frac{sp}{p-1}j} \frac{M_l(4^{j+1}\rho) - M_l(4\rho)}{\rho} {\bf 1}_{\lbrace 4^j \rho < r_{\ast} \rbrace} \right) \mathrm{d}\rho \\
&= \sum_{j=1}^{\infty} \left( 2^{-\frac{sp}{p-1}j} \int_0^{r_{\ast}/4^j} \frac{M_l(4^{j+1}\rho) - M_l(4\rho)}{\rho} \,\mathrm{d}\rho \right) \\
&\leq \sum_{j=1}^{\infty} \left( 2^{-\frac{sp}{p-1}j} \int_{r_{\ast}/4^j}^{r_{\ast}} \frac{M_l(4\rho)}{\rho}  \,\mathrm{d}\rho \right) \\
&\leq (\log 4) \left( \sum_{j=1}^{\infty} j 2^{-\frac{sp}{p-1}j} \right) \left( \wsup_{B_{4r_{\ast}}(x_0)} u+|l| \right) \\
&\leq C \left( \wsup_{B_{4r_{\ast}}(x_0)} u+|l| \right),
\end{split}
\end{equation*}
and consequently,
\begin{equation} \label{eq:Wiener-I2}
I_2 \leq C \left( \mathrm{Tail}(u; x_0, r_{\ast}) + \wsup_{B_{4r_{\ast}}(x_0)} u + |l| \right) < +\infty.
\end{equation}
Note that the left-hand side of \eqref{eq:Wiener-I12} diverges by the assumption of the sufficient condition, while the right-hand side is finite by \eqref{eq:Wiener-I1} and \eqref{eq:Wiener-I2}. This leads us to a contradiction, and hence the first inequality in \eqref{eq:contrad} holds. The second inequality in \eqref{eq:contrad} can be proved in the same way, so we conclude \eqref{eq:limit}.
\end{proof}


\section{The necessary condition for the regularity of a boundary point} \label{sec:necessary}


In this section, the necessary condition is established for a boundary point to be regular with respect to $\mathcal{L}$. We develop a dual approach to $(s, p)$-capacity as a nonlocal analog of the approach in \cite{KM94}.

Let $\Omega$ be a bounded open set and $K \subset \Omega$ a compact set. Let $u = \mathfrak{R}(K, \Omega)$ be the $\mathcal{L}$-potential of $K$ in $\Omega$. Then, we recall that, by \Cref{lem:potential} (iv), $u$ is a weak supersolution of $\mathcal{L}u = 0$ in $\Omega$, meaning that $\mathcal{E}^k(u, \varphi) \geq 0$ for every nonnegative function $\varphi \in C_c^{\infty}(\Omega)$. Thus, there exists a unique nonnegative Radon measure $\mu_u$ in $\Omega$ such that $\mathcal{L}u=\mu_u$ in $\Omega$ in the sense of distribution, that is,
\begin{equation*}
\langle \mu_u, \varphi \rangle := \int_{\Omega} \varphi \,\mathrm{d}\mu_u =\mathcal{E}^k(u, \varphi)
\end{equation*}
for every $\varphi \in C_c^{\infty}(\Omega)$. We call this measure $\mu_u$ the \textit{$\mathcal{L}$-distribution of $K$ with respect to $\Omega$}. Again, by \Cref{lem:potential} (i), we have $u \in W_0^{s, p}(\Omega)$, which implies that  $\mu_u \in (W^{s, p}_0(\Omega))^{\ast}$.

\begin{lemma}\label{mecap}
Let $\Omega$ be a bounded open set and $K \subset \Omega$ be a compact set. If $u = \mathfrak{R}(K, \Omega)$ is the $\mathcal{L}$-potential of $K$ in $\Omega$ with the $\mathcal{L}$-distribution $\mu_u$, then
	\begin{align*}
	\mu_u(E) \leq \Lambda \mathrm{cap}_{s,p} (K \cap E, \Omega)
	\end{align*}
	for every compact set $E \subset \Omega$.
\end{lemma}

\begin{proof}
Let $\nu$ be the measure obtained by restricting the $\mathcal{L}$-distribution $\mu_u$ to the compact set $E$. We know that $\mu_u, \nu \in (W^{s, p}_0(\Omega))^{\ast}$. Let $v \in W^{s, p}_0(\Omega)$ be the unique weak solution of $\mathcal{L}v=\nu$ in $\Omega$ guaranteed by \Cref{thm:DP}. Then, $0 \leq v \leq u \leq 1$ a.e. in $\mathbb{R}^n$ by \Cref{thm:comparison}.

	Let $\varphi \in C_c^{\infty}(\Omega)$ be such that $\varphi \equiv 1$ on the compact set $K\cap E$. Since $\mu_u$ and $\nu$ are supported on $K$ and $K \cap E$, respectively, we obtain
	\begin{equation} \label{munu2}
	\mu_u(E) = \nu(E) = \int_{\Omega}\varphi \,\mathrm{d}\nu = \mathcal{E}^k(v, \varphi) \leq \frac{p-1}{p} \mathcal{E}^k(v, v) + \frac{1}{p} \mathcal{E}^{k}(\varphi, \varphi),
	\end{equation}
	where we used Young's inequality in the last inequality. To estimate $\mathcal{E}^k(v, v)$, we consider a test function 
	\begin{equation*}
	\overline{\varphi}(x):=\max \lbrace\varphi(x), v(x) \rbrace \in W_0^{s,p}(\Omega).
	\end{equation*}
Note that $\overline{\varphi} \equiv \varphi \equiv 1$ on $K \cap E$ and $\overline{\varphi} - \varphi \in W_{0}^{s,p}(\Omega \setminus (K\cap E))$. Moreover, by definition of $\overline{\varphi}$, it is immediate that $\overline{\varphi}-v \in W_{0}^{s,p}(\Omega)$ and that $\overline{\varphi}-v$ is nonnegative in $\Omega$.
	
	Recalling that $v$ is a weak supersolution of $\mathcal{L}v = 0$ in $\Omega$ and a weak solution of $\mathcal{L}v = 0$ in $\Omega \setminus (K\cap E)$, we have
	\begin{align*}
	\mathcal{E}^{k}(v, \varphi-v)=\mathcal{E}^{k}(v, \varphi-\overline{\varphi})+\mathcal{E}^{k}(v, \overline{\varphi}-v)\geq 0.
	\end{align*}
	Hence, we obtain
	\begin{align}\label{munu3}
	\mathcal{E}^k(v, v) \leq \mathcal{E}^k(v, \varphi).
	\end{align}
	By combining two estimates \eqref{munu2} and \eqref{munu3}, we conclude that
	\begin{align*}
	\mu_u(E) \leq \mathcal{E}^k(\varphi, \varphi) \leq \Lambda \int_{\mathbb{R}^n}\int_{\mathbb{R}^n}\frac{|\varphi(x)-\varphi(y)|^{p}}{|x-y|^{n+sp}} \,\mathrm{d}y \,\mathrm{d}x
	\end{align*}
	for any $\varphi \in C_c^{\infty}(\Omega)$ with $\varphi \equiv 1$ on $K \cap E$. We complete the proof by taking the infimum over all such $\varphi$.
\end{proof}

The necessary part of the proof of \Cref{thm:main} is based on pointwise estimates of $\mathcal{L}$-superharmonic functions in terms of Wolff potentials. We recall from \cite{KMS15} the definition of the Wolff potential with fractional orders.

\begin{definition}[Wolff potential]
Let $s \in (0,1)$ and $p \in (1, \infty)$. Let $\mu$ be a nonnegative Borel measure with finite total mass on $\mathbb{R}^n$. The {\it Wolff potential ${\bf W}^{\mu}_{s, p}$ of the measure $\mu$} is defined by 
	\begin{align*}
	{\bf W}^{\mu}_{s, p}(x_0, r) := \int_0^r \left( \frac{\mu(B_{\rho}(x_0))}{\rho^{n-sp}}\right)^{\frac{1}{p-1}} \frac{\mathrm{d}\rho}{\rho}, 
	\end{align*}
	whenever $x_0 \in \mathbb{R}^n$ and $0<r\leq \infty$.
\end{definition}

We now present a pointwise estimate for nonnegative $\mathcal{L}$-superharmonic functions in $\Omega$ in terms of the Wolff potential. It will be stated in a slightly different context, compared to the potential estimate established in \cite{KMS15}. Note that estimates in \cite{KMS15} are obtained for SOLA (Solutions Obtained as Limits of Approximations), which requires $p > 2-s/n$. Our alternative approach is more suitable for proving the necessity of the Wiener criterion for all ranges of $p \in (1, n/s]$; see also \cite{AK04,CGZG20,KM92,KM94,KK10,TW02} for similar results concerning local operators.

\begin{theorem} \label{thm:Wolff}
Assume that $p \in (1, n/s]$. Let $u$ be an $\mathcal{L}$-superharmonic function in $B_{4r}(x_0)$ which is nonnegative in $B_{4r}(x_0)$. If $\mu = \mathcal{L}u$ exists, then
\begin{equation*}
u(x_0) \leq C \left( \inf_{B_{r}(x_0)}u + {\bf W}_{s,p}^{\mu}(x_0, 2r) + \mathrm{Tail}(u; x_0, r) \right)
\end{equation*}
for some $C = C(n, s, p, \Lambda) > 0$.
\end{theorem}

Note that the value $u(x_0)$ in \Cref{thm:Wolff} is defined as in \Cref{thm:KKP17}. To prove \Cref{thm:Wolff}, we need the following auxiliary estimate.

\begin{lemma}\label{auxest}
Assume that $p \in (1, n/s]$. Let $u$ be an $\mathcal{L}$-superharmonic function in $B_{2r}(x_0)$ and suppose that $\mu = \mathcal{L}u$ exists. Assume that
\begin{equation} \label{eq:gamma}
p-1 < \gamma < \frac{n(p-1)}{n-s(p-1)},
\end{equation}
then there exists $C = C(n, s, p, \Lambda, \gamma) > 0$ such that, for any $a \in \mathbb{R}$ and $d > 0$,
\begin{equation} \label{eq:auxest}
\begin{split}
\left( d^{-\gamma} \fint_{B_r(x_0)} (u(x)-a)_+^{\gamma} \,\mathrm{d}x \right)^{p/q}
&\leq C d^{-\gamma} \fint_{B_{2r}(x_0)} (u(x)-a)_+^{\gamma} \,\mathrm{d}x + C d^{1-p} \frac{\mu(B_{2r}(x_0))}{r^{n-sp}} \\
&\quad + C d^{1-p} \mathrm{Tail}^{p-1}((u-a)_+; x_0, 2r),
\end{split}
\end{equation}
provided that
\begin{equation} \label{eq:hypothesis}
|B_{2r}(x_0) \cap \lbrace u > a \rbrace| < \frac{1}{2} d^{-\gamma} \int_{B_r(x_0)} (u(x)-a)_+^{\gamma} \,\mathrm{d}x,
\end{equation}
where $q = p\gamma/(p-\gamma/(p-1))$.
\end{lemma}

\begin{proof}
It is enough to prove the lemma for locally bounded function $u$. Indeed, an $\mathcal{L}$-superharmonic function $u$ can be approximated by a sequence of bounded $\mathcal{L}$-superharmonic functions $u_j := \min\lbrace u, j \rbrace$, $j \geq 1$. By the argument as in the beginning of this section, there exists a unique nonnegative Radon measure $\mu_{u_j}$ in $\Omega$ such that $\mathcal{L}u_j = \mu_{u_j}$ in $\Omega$ in the sense of distribution. Then, the hypothesis \eqref{eq:hypothesis} holds for $u_j$ for sufficiently large $j$. Applying the result to $u_j$, we obtain \eqref{eq:auxest} with $u$ and $\mu$ replaced by $u_j$ and $\mu_{u_j}$, respectively. Therefore, by passing the limit $j \to \infty$ and using the weak convergence $\mu_{u_j} \rightharpoonup \mu$, we conclude the lemma for general case.

For simplicity, we write $B_r = B_r(x_0)$ and assume that $a=0$. Let us assume that $u$ is locally bounded. Then, by \Cref{thm:KKP17}, we have $u \in W^{s, p}_{\mathrm{loc}}(B_{2r})$. Note that \eqref{eq:gamma} is equivalent to $p < q < np/(n-sp)$. Here, $np/(n-sp) = \infty$ when $p=n/s$. By the hypothesis \eqref{eq:hypothesis}, we obtain
\begin{equation*}
d^{-\gamma} \int_{B_{r} \cap \lbrace 0 < u < d \rbrace} u^{\gamma}(x) \,\mathrm{d}x \leq |B_r \cap \lbrace 0 < u < d \rbrace| \leq |B_{2r} \cap \lbrace u > 0 \rbrace| \leq \frac{1}{2} d^{-\gamma} \int_{B_r} u_+^{\gamma}(x) \,\mathrm{d}x,
\end{equation*}
and hence
\begin{equation*}
d^{-\gamma} \fint_{B_{r}} u_+^{\gamma}(x) \,\mathrm{d}x \leq 2 d^{-\gamma} \frac{1}{|B_r|} \int_{B_{r} \cap \lbrace u \geq d \rbrace} u_+^{\gamma}(x) \,\mathrm{d}x \leq C \fint_{B_r} w^q(x) \,\mathrm{d}x,
\end{equation*}
where
\begin{equation*}
w = \left( 1 + \frac{u_+}{d} \right)^{\gamma/q}-1.
\end{equation*}
Let $\eta \in C^{\infty}_{c}(B_{3r/2})$ be such that $0 \leq \eta \leq 1$, $\eta = 1$ on $B_r$, and $|\nabla \eta| \leq C/r$. If $p \in (1, n/s)$, then the fractional Sobolev inequality yields
\begin{equation} \label{eq:subcritical}
\begin{split}
\left( \fint_{B_r} w^q \,\mathrm{d}x \right)^{p/q}
&\leq C \left( \fint_{B_{2r}} (w\eta)^q \,\mathrm{d}x \right)^{p/q} \leq C \left( \fint_{B_{2r}} (w\eta)^{p^{\ast}} \,\mathrm{d}x \right)^{p/p^{\ast}} \\
&\leq Cr^{sp-n} \int_{B_{2r}} \int_{B_{2r}} \frac{|w(x)\eta(x)-w(y)\eta(y)|^p}{|x-y|^{n+sp}} \,\mathrm{d}y \,\mathrm{d}x + C \fint_{B_{2r}} w^p \,\mathrm{d}x \\
&=: I_1 + I_2.
\end{split}
\end{equation}
If $p = n/s$, then let $t \in [1, p)$ and $\sigma \in (0,s)$ be such that $q < t^{\ast}_{\sigma} := nt/(n-\sigma t)$. In this case, we replace $p^{\ast}$ by $t^{\ast}_{\sigma}$ in \eqref{eq:subcritical} and then apply the fractional Sobolev inequality with $t < n/\sigma$, \cite[Lemma 4.6]{Coz17}, and H\"older's inequality to have the estimate. 

We define a test function
\begin{equation*}
v = \left( 1 - \left( 1+ \frac{u_+}{d} \right)^{1-\tau} \right) \eta^p, \quad \tau = \frac{\gamma}{p-1}.
\end{equation*}
Then, it follows from $\mathcal{E}^k(u, v) = \langle \mu, v \rangle$ that
\begin{equation*}
\begin{split}
&\int_{B_{2r}} \int_{B_{2r}} \left| \frac{u(x)}{d}-\frac{u(y)}{d} \right|^{p-2}\left( \frac{u(x)}{d}-\frac{u(y)}{d} \right)(v(x)-v(y)) k(x, y) \,\mathrm{d}y \,\mathrm{d}x \\
&= d^{1-p} \int_{B_{2r}} v(x) \,\mathrm{d}\mu(x) - 2d^{1-p} \int_{B_{2r}} \int_{\mathbb{R}^n \setminus B_{2r}} |u(x)-u(y)|^{p-2}(u(x)-u(y))v(x) k(x, y) \,\mathrm{d}y \,\mathrm{d}x.
\end{split}
\end{equation*}
By applying \Cref{lem:alg-ineq-wolff} to $a=u(x)/d$, $b=u(y)/d$, $\eta_1=\eta(x)$, and $\eta_2=\eta(y)$, we obtain
\begin{equation*}
\begin{split}
I_1
&\leq Cr^{sp-n} \int_{A^+_{2r}} \left( 1+\frac{u_+(x)}{d} \right)^{\gamma} \int_{A^+_{2r}} |\eta(x)-\eta(y)|^p k(x, y) \,\mathrm{d}y \,\mathrm{d}x \\
&\quad + Cd^{1-p} r^{sp-n} \int_{B_{2r}} v(x) \,\mathrm{d}\mu(x) \\
&\quad - C d^{1-p} r^{sp-n} \int_{B_{2r}} \int_{\mathbb{R}^n \setminus B_{2r}} |u(x)-u(y)|^{p-2}(u(x)-u(y)) v(x) k(x, y) \,\mathrm{d}y \,\mathrm{d}x \\
&\leq C r^{-n} \int_{A^+_{2r}} \left( 1+\frac{u_+(x)}{d} \right)^{\gamma} \,\mathrm{d}x + Cd^{1-p} \frac{\mu(B_{2r})}{r^{n-sp}} \\
&\quad + C d^{1-p} r^{sp-n} \int_{B_{2r}} \int_{\mathbb{R}^n \setminus B_{2r}} \frac{(u(y)-u(x))_+^{p-1} \eta^p(x)}{|x-y|^{n+sp}} \,\mathrm{d}y \,\mathrm{d}x,
\end{split}
\end{equation*}
where $A_{2r}^+ = B_{2r} \cap \lbrace u > 0 \rbrace$. By the hypothesis \eqref{eq:hypothesis}, we have
\begin{equation*}
r^{-n} \int_{A^+_{2r}} \left( 1+\frac{u_+(x)}{d} \right)^{\gamma} \,\mathrm{d}x \leq C d^{-\gamma} \fint_{B_{2r}} u_+^{\gamma}(x) \,\mathrm{d}x.
\end{equation*}
Moreover, recalling that $\mathrm{\supp} \,\eta \subset B_{3r/2}$, we obtain
\begin{equation*}
\begin{split}
&C d^{1-p} r^{sp-n} \int_{B_{2r}} \int_{\mathbb{R}^n \setminus B_{2r}} \frac{(u(y)-u(x))_+^{p-1} \eta^p(x)}{|x-y|^{n+sp}} \,\mathrm{d}y \,\mathrm{d}x \\
&\leq C d^{1-p} r^{sp-n} \int_{B_{3r/2}} \int_{\mathbb{R}^n \setminus B_{2r}} \frac{u_+^{p-1}(y)}{|y-x_0|^{n+sp}} \,\mathrm{d}y \,\mathrm{d}x \\
&\leq C d^{1-p} \mathrm{Tail}^{p-1}(u_+; x_0, 2r).
\end{split}
\end{equation*}

We next estimate $I_2$. Since $w^q \leq (1+u/d)^{\gamma}$, by H\"older's inequality and the hypothesis \eqref{eq:hypothesis} we have
\begin{equation*}
I_2 \leq \frac{C}{r^{n}} \int_{A^+_{2r}} w^p(x) \,\mathrm{d}x \leq \frac{C}{r^n} |A^+_{2r}|^{1-p/q} \left( \int_{A^+_{2r}} w^q(x) \,\mathrm{d}x \right)^{p/q} \leq C d^{-\gamma} \fint_{B_{2r}} u_+^{\gamma}(x) \,\mathrm{d}x.
\end{equation*}
By combining all the estimates, we conclude the lemma.
\end{proof}

\begin{proof} [Proof of \Cref{thm:Wolff}]
	We fix a constant $\delta \in (0, 1)$ to be determined later. Let $B_j=B_{r_j}(x_0)$, where $r_j=2^{1-j}r$. We define a sequence $\{a_j\}_{j=0}^{\infty}$ recursively. Let $a_0=0$ and for $j \geq 0$, let 
	\begin{align} \label{eq:def-aj}
	a_{j+1}:=a_j+\delta^{-1}\left( \fint_{B_{j+1}} (u(x)-a_j)_+^{\gamma} \,\mathrm{d}x\right)^{1/\gamma},
	\end{align}
	where $\gamma$ satisfies \eqref{eq:gamma}. We will prove the desired estimate in several steps.

	We first prove the following estimate: for $q = p\gamma/(p-\gamma/(p-1))$, 
		\begin{align} \label{estlem}
		\delta^{p\gamma/q} \leq C \delta^{\gamma}\left(\frac{a_j-a_{j-1}}{a_{j+1}-a_j}\right)^{\gamma}+C (a_{j+1}-a_j)^{1-p} \left( \frac{\mu(B_j)}{r_j^{n-sp}}+ \mathrm{Tail}^{p-1}((u-a_j)_+; x_0, r_j) \right)
		\end{align}
		whenever $a_{j+1} >a_j$, $j \geq 1$. We observe that
		\begin{align}
		\begin{split} \label{eq:vol-aj}
		|B_j \cap \{u>a_j\}| &\leq (a_j-a_{j-1})^{-\gamma} \int_{B_j \cap \{u>a_j\}}(u(x)-a_{j-1})^\gamma \, \mathrm{d}x \\
		&\leq (a_j-a_{j-1})^{-\gamma} \int_{B_j}(u(x)-a_{j-1})_+^\gamma \, \mathrm{d}x \\
		&=\delta^{\gamma} |B_{j}| = 2^n \delta^{\gamma} |B_{j+1}| \\
		&=2^n(a_{j+1}-a_j)^{-\gamma} \int_{B_{j+1}}(u(x)-a_{j})_+^\gamma \, \mathrm{d}x.
		\end{split}
		\end{align}
		On the one hand, by taking $\delta^{\gamma} \leq 1/2$, we have $|B_j \cap \lbrace u > a_j \rbrace| \leq \frac{1}{2} |B_j|$, which implies that
		\begin{align} \label{inf}
			\inf_{B_j}u \leq a_j.
		\end{align}
		On the other hand, if we choose 
		\begin{align*}
		d_j:=2^{-(n+2)/\gamma}(a_{j+1}-a_j)>0,
		\end{align*}
		then it follows that 
		\begin{align*}
			|B_{j} \cap \{u>a_j\}| < \frac{1}{2}d_j^{-\gamma} \int_{B_{j+1}}(u(x)-a_j)_+^{\gamma} \,\mathrm{d}x,
		\end{align*}
		which is the hypothesis of \Cref{auxest}. Therefore, by \Cref{auxest}, we derive
	\begin{equation*}
	\begin{split}
	\delta^{p\gamma/q}
	&= C \left( d_j^{-\gamma} \fint_{B_{{j+1}}} (u(x)-a_j)_+^{\gamma} \,\mathrm{d}x \right)^{p/q} \\
	&\leq C d_j^{-\gamma} \fint_{B_{j}} (u(x)-a_j)_+^{\gamma} \,\mathrm{d}x +C d_j^{1-p} \left( \frac{\mu(B_{j})}{r_j^{n-sp}}+ \mathrm{Tail}^{p-1}((u-a_j)_+; x_0, r_j) \right).
	\end{split}
	\end{equation*}
	Recalling that
	\begin{align*}
	d_j^{-\gamma} \fint_{B_{j}} (u(x)-a_j)_+^{\gamma} \,\mathrm{d}x \leq  d_j^{-\gamma} \fint_{B_{j}} (u(x)-a_{j-1})_+^{\gamma} \,\mathrm{d}x= C \delta^{\gamma} (d_{j-1}/d_j)^{\gamma},
	\end{align*}
	we conclude that the estimate \eqref{estlem} holds.
	
We next prove the following recursive inequality using \eqref{estlem}: for $j \geq 1$,
\begin{equation} \label{induction}
\begin{split}
d_j
&\leq \frac{1}{2} d_{j-1} + C\delta^{\frac{\gamma}{p-1}(1-\frac{p}{q})} \sum_{i=2}^{j} 2^{-\frac{sp}{2(p-1)}i} d_{j-i} \\
&\quad + C \delta^{-\frac{p\gamma}{(p-1)q}} \left( \left( \frac{\mu(B_j)}{r_j^{n-sp}}\right)^{\frac{1}{p-1}} + 2^{-\frac{sp}{p-1}j} \mathrm{Tail}(u_+; x_0, r) \right).
\end{split}
\end{equation}
	We may assume that $d_j > \frac{1}{2} d_{j-1}$; otherwise, the inequality becomes trivial. Then by \eqref{estlem}, we have
	\begin{align*}
	1 \leq C \delta^{\gamma(1-p/q)}+C \delta^{-p\gamma/q} d_j^{1-p} \left( \frac{\mu(B_j)}{r_j^{n-sp}}+ \mathrm{Tail}^{p-1}((u-a_j)_+; x_0, r_j) \right).
	\end{align*}
	By choosing $0<\delta=\delta(n,s,p,\Lambda, \gamma)<1$ sufficiently small so that $C\delta^{\gamma(1-p/q)} \leq 1/2$, we obtain
	\begin{equation*}
	d_j \leq C \delta^{-\frac{p\gamma}{(p-1)q}} \left( \left( \frac{\mu(B_j)}{r_j^{n-sp}}\right)^{\frac{1}{p-1}} + \mathrm{Tail}((u-a_j)_+; x_0, r_j) \right).
	\end{equation*}
Since
	\begin{equation*}
	\begin{split}
	&\mathrm{Tail}^{p-1}((u-a_j)_+; x_0, r_j) \\
	&= \left( \frac{r_j}{r_1} \right)^{sp} \mathrm{Tail}^{p-1}((u-a_j)_+; x_0, r) + r_j^{sp} \int_{B_1 \setminus B_j} \frac{(u(x)-a_j)_+^{p-1}}{|x-x_0|^{n+sp}} \,\mathrm{d}x \\
	&\leq C 2^{-spj} \mathrm{Tail}^{p-1}(u_+; x_0, r) + r_j^{sp} \int_{B_{1} \setminus B_{j}} \frac{(u(x)-a_j)_+^{p-1}}{|x-x_0|^{n+sp}} \,\mathrm{d}x,
	\end{split}
	\end{equation*}
	it is enough to show that
	\begin{equation} \label{eq:I}
	I:= \left( r_j^{sp} \int_{B_{1} \setminus B_{j}} \frac{(u(x)-a_j)_+^{p-1}}{|x-x_0|^{n+sp}} \,\mathrm{d}x \right)^{\frac{1}{p-1}} \leq C\delta^{\frac{\gamma}{p-1}} \sum_{i=2}^{j} 2^{-\frac{sp}{2(p-1)}i} d_{j-i}.
	\end{equation}
Indeed, by using \eqref{eq:Dyd06} with $t = 1/(p-1)$ and $\beta = 2^{\frac{sp}{2(p-1)}} > 1$, and then using H\"older's inequality, we have
\begin{equation*}
\begin{split}
I
&\leq \left( \sum_{i=2}^{j} \frac{r_j^{sp}}{r_{j-i+2}^{n+sp}} \int_{B_{j-i+1} \setminus B_{j-i+2}} (u(x)-a_j)_+^{p-1} \,\mathrm{d}x \right)^{\frac{1}{p-1}} \\
&\leq C \sum_{i=2}^{j} 2^{-\frac{sp}{2(p-1)}i} r_{j-i+2}^{-\frac{n}{p-1}} \left( \int_{B_{j-i+1}} (u(x)-a_j)_+^{p-1} \,\mathrm{d}x \right)^{\frac{1}{p-1}} \\
&\leq C \sum_{i=2}^{j} 2^{-\frac{sp}{2(p-1)}i} r_{j-i+2}^{-\frac{n}{p-1}} \left( \int_{B_{j-i+1}} (u(x)-a_j)_+^{\gamma} \,\mathrm{d}x \right)^{\frac{1}{\gamma}} |B_{j-i+1} \cap \lbrace u > a_{j} \rbrace|^{\frac{1}{p-1}-\frac{1}{\gamma}}.
\end{split}
\end{equation*}
Moreover, by \eqref{eq:vol-aj} and \eqref{eq:def-aj} we deduce
\begin{equation*}
\begin{split}
I
&\leq C \sum_{i=2}^{j} 2^{-\frac{sp}{2(p-1)}i} r_{j-i+2}^{-\frac{n}{p-1}} \left( \int_{B_{j-i+1}} (u(x)-a_{j-i})_+^{\gamma} \,\mathrm{d}x \right)^{\frac{1}{\gamma}} |B_{j-i+1} \cap \lbrace u > a_{j-i+1} \rbrace|^{\frac{1}{p-1}-\frac{1}{\gamma}} \\
&\leq C \sum_{i=2}^{j} 2^{-\frac{sp}{2(p-1)}i} r_{j-i+2}^{-\frac{n}{p-1}} d_{j-i}^{-\frac{\gamma}{p-1}+1} \left( \int_{B_{j-i+1}} (u(x)-a_{j-i})_+^{\gamma} \,\mathrm{d}x \right)^{\frac{1}{p-1}} \\
&\leq C \delta^{\frac{\gamma}{p-1}} \sum_{i=2}^{j} 2^{-\frac{sp}{2(p-1)}i} d_{j-i},
\end{split}
\end{equation*}
which proves \eqref{eq:I}.

	Finally, we are ready to finish the proof. By an iterative application of \eqref{induction}, we have
	\begin{align*}
	a_k-a_1 &\leq a_{k+1}-a_1 =2^{(n+2)/\gamma}\sum_{j=1}^k d_j \\
	&\leq \frac{1}{2}a_k + C \delta^{\frac{\gamma}{p-1}(1-\frac{p}{q})} \sum_{j=1}^{k} \sum_{i=2}^{j} 2^{-\frac{sp}{2(p-1)}i} d_{j-i} \\
	&\quad + C \delta^{-\frac{p\gamma}{(p-1)q}} \sum_{j=1}^k \left( \left( \frac{\mu(B_j)}{r_j^{n-sp}}\right)^{\frac{1}{p-1}} + 2^{-\frac{sp}{p-1}j} \mathrm{Tail}(u_+; x_0, r) \right).
	\end{align*}
	Note that we have
	\begin{equation*}
	\sum_{j=1}^{k} \sum_{i=2}^{j} 2^{-\frac{sp}{2(p-1)}i} d_{j-i} = \sum_{i=2}^{k} 2^{-\frac{sp}{2(p-1)}i} \sum_{j=i}^{k} d_{j-i} \leq 2^{-(n+2)/\gamma}\sum_{i=2}^{k} 2^{-\frac{sp}{2(p-1)}i} a_k \leq Ca_k.
	\end{equation*}
	Thus, we take $\delta$ sufficiently small so that
	\begin{equation*}
	a_k-a_1 \leq \frac{3}{4} a_k + C\,\sum_{j=1}^k \left(
	\frac{\mu(B_j)}{r_j^{n-sp}}\right)^{\frac{1}{p-1}}+ C \left( \sum_{j=1}^{k} 2^{-\frac{sp}{p-1}j} \right) \mathrm{Tail}(u_+; x_0, r).
	\end{equation*}
	In other words,
	\begin{align*}
	\limsup_{k \to \infty}a_k
	&\leq C \left( \fint_{B_{r}(x_0)}u^{\gamma} \,\mathrm{d}x \right)^{1/\gamma} + C \, {\bf W}_{s,p}^{\mu_u}(x_0, 2r)+ C \, \mathrm{Tail}(u_+; x_0, r) \\
	&\leq C \left( \fint_{B_{2r}(x_0)}u^{\gamma} \,\mathrm{d}x \right)^{1/\gamma} + C \, {\bf W}_{s,p}^{\mu_u}(x_0, 2r)+ C \, \mathrm{Tail}(u_+; x_0, r).
	\end{align*}
	Recalling the lower semicontinuity of $u$, together with \eqref{inf}, we have
	\begin{align*}
		u(x_0) \leq \lim_{k \to \infty}\inf_{B_k}u \leq \liminf_{k \to \infty}a_k.
	\end{align*}
	Since $\gamma < \frac{n(p-1)}{n-s(p-1)} < \frac{n(p-1)}{n-sp}$, \Cref{thm:interior-WHI} completes the proof.
\end{proof}

We turn back to the necessary part of the Wiener criterion. Let $x_0 \in \partial \Omega$ be a boundary point and assume that $p \leq n/s$. Let $u_\rho = \mathfrak{R}(D_\rho(x_0), B_{8\rho}(x_0))$ be the $\mathcal{L}$-potential of $D_\rho(x_0)$ in $B_{8\rho}(x_0)$, where $D_{\rho}(x_0)$ is defined as \eqref{eq:D}. Let $\mu_\rho$ be the $\mathcal{L}$-distribution of $D_\rho(x_0)$ with respect to $B_{8\rho}(x_0)$.

\begin{lemma} \label{irr}
	If there exists $\rho>0$ such that
	\begin{align} \label{irreq}
	u_{\rho}(x_0)=\liminf_{\Omega \ni x \to x_0}u_{\rho}(x) <1,
	\end{align}
	then the boundary point $x_0$ is irregular.
\end{lemma}

\begin{proof}
Note that, if \eqref{irreq} holds for some $\rho > 0$, then it holds for $\rho' < \rho$ as well by \Cref{thm:comparison}. For simplicity, we write $B_\rho = B_{\rho}(x_0)$ and $D_\rho = D_{\rho}(x_0)$. Let us choose $\rho_0>0$ sufficiently small so that $\Omega\cap \partial B_{8\rho} \neq \emptyset$ for all $\rho < \rho_0$. We fix $\rho < \rho_0$ and let $r < \rho$. Let $g \in C^{\infty}_c(\mathbb{R}^n)$ be a function such that $g(x)=3/2$ on $\overline{B_{r/2}} \setminus \Omega$, $0 \leq g(x)<3/2$ in $(B_r \setminus \overline{B_{r/2}})\setminus \Omega$, and $g(x)=0$ on the remaining part of $\mathbb{R}^n \setminus \Omega$. Let us denote by $v_r$ the $\mathcal{L}$-harmonic function in $\Omega$ such that $v_r-g \in W_{0}^{s,p}(\Omega)$. We claim that there exists $r_0 = r_0(n, s, p, \rho, \Lambda, \mathrm{diam}(\Omega)) > 0$ such that
	\begin{align} \label{wts0}
	v_r \leq \frac{1}{2}+u_{\rho} \quad\text{in } \Omega\cap B_{8\rho}
	\end{align}
for all $r < r_0$. Once \eqref{wts0} is proved, then we deduce
	\begin{align*}
	\liminf_{\Omega \ni x \to x_0} v_r <\frac{3}{2}=g(x_0),
	\end{align*}
	which concludes that the boundary point $x_0$ is irregular. 
	
The claim \eqref{wts0} can be proved by using the comparison principle and the fact that $v_r, u_\rho$ are continuous in $\Omega \cap B_{8\rho}$. To this end, let us check that $v_r \leq 1/2+u_{\rho}$ in $\mathbb{R}^n \setminus (\Omega \cap B_{8\rho})$ in the sense of $V^{s, p}(\Omega \cap B_{8\rho}|\mathbb{R}^n)$. We divide $\mathbb{R}^n \setminus (\Omega \cap B_{8\rho})$ into three regions;
\begin{equation*}
\mathbb{R}^n \setminus (\Omega \cap B_{8\rho}) = D_{\rho} \cup \left( (\mathbb{R}^n \setminus \overline{B_{\rho}}) \setminus \Omega \right) \cup (\Omega \setminus B_{8\rho}).
\end{equation*}
Let $f \in C^{\infty}_{c}(B_{8\rho})$ be such that $0 \leq f \leq 1$ and $f = 1$ on $D_{\rho}$. Then, $u_\rho = f$ in $D_\rho \cup (\mathbb{R}^n \setminus B_{8\rho})$ in the sense of $V^{s, p}(B_{8\rho} \setminus D_{\rho}|\mathbb{R}^n)$. Since $\Omega \cap B_{8\rho} \subset B_{8\rho} \setminus D_{\rho}$, we have
\begin{equation*}
v_r = g \leq \frac{3}{2} = \frac{1}{2} + f = \frac{1}{2} + u_{\rho} \quad\text{in } D_{\rho}
\end{equation*}
and
\begin{equation*}
v_r = g = 0 < \frac{1}{2} \leq \frac{1}{2} + u_{\rho} \quad\text{in } (\mathbb{R}^n \setminus \overline{B_{\rho}}) \setminus \Omega
\end{equation*}
in the sense of $V^{s, p}(\Omega \cap B_{8\rho}|\mathbb{R}^n)$. Therefore, it is sufficient to establish the estimate $v_r \leq \frac{1}{2} + u_{\rho}$ in $\Omega \setminus B_{8\rho}$ in the sense of $V^{s, p}(\Omega \cap B_{8\rho}|\mathbb{R}^n)$. Since $u_{\rho} = 0$ in $\mathbb{R}^n \setminus B_{8\rho}$ in the sense of $V^{s, p}(\Omega \cap B_{8\rho}|\mathbb{R}^n)$, we will prove that
\begin{equation} \label{eq:claim2}
v_r \leq \frac{1}{2} \quad\text{in } \Omega \setminus B_{8\rho} \text{ in the sense of } V^{s, p}(\Omega \cap B_{8\rho}|\mathbb{R}^n)
\end{equation}
in the rest of the proof.

Let $z \in \Omega \setminus B_{8\rho}(x_0)$ and consider a ball $B_{\rho}(z)$. Since $B_r(x_0) \cap B_{\rho}(z) = \emptyset$ and $g=0$ in $(\mathbb{R}^n \setminus B_r(x_0)) \setminus \Omega$, we have $g=0$ in $B_{\rho}(z) \setminus \Omega$. By applying \Cref{thm:loc-bdd} to $v_r$ in $B_{\rho}(z)$, we obtain
\begin{equation} \label{lb}
\wsup_{B_{\rho/2}(z)} v_r \leq \delta \,\mathrm{Tail}(v_r; z, \rho/2) + C \delta^{-\frac{(p-1)n}{sp^2}} \left( \fint_{B_{\rho}(z)} v_r^p(x) \,\mathrm{d}x \right)^{1/p}.
\end{equation}
To estimate the terms on the right-hand side of \eqref{lb}, we let $w$ be the $\mathcal{L}$-potential of $D_r(x_0)$ in $B_R(x_0)$, where $R$ is the diameter of $\Omega$. By the comparison principle, $v_r(x) \leq \frac{3}{2}w(x)$ a.e. in $\mathbb{R}^n$. Since $w \leq 1$, we have
\begin{equation} \label{eq:vr-tail}
\mathrm{Tail}(v_r; z, \rho/2) \leq \frac{3}{2} \mathrm{Tail}(w; z, \rho/2) \leq \frac{3}{2} \left( \left( \frac{\rho}{2} \right)^{sp} \int_{\mathbb{R}^n \setminus B_{\rho/2}(z)} \frac{1}{|y-z|^{n+sp}} \,\mathrm{d}y \right)^{\frac{1}{p-1}} \leq C.
\end{equation}
For the last term on the right-hand side of \eqref{lb}, we apply the fractional Poincar\'e inequality, \Cref{poincare}, to obtain 
\begin{align}\label{eq:poincare}
	\int_{B_{\rho}(z)} v_r^p(x) \,\mathrm{d}x \leq (3/2)^p \int_{B_R(x_0)} w^p(x) \,\mathrm{d}x \leq C(n,s,p,\Lambda) R^{sp} \, \mathcal{E}^{k}(w,w).
\end{align}
Moreover, we observe from \Cref{lem:cap} that
\begin{equation} \label{eq:capacity}
\begin{split}
\mathcal{E}^k(w, w)
&\leq \Lambda \mathrm{cap}_{s, p}(D_r(x_0), B_R(x_0)) \\
&\leq \Lambda \mathrm{cap}_{s, p}(\overline{B_r(x_0)}, B_{R}(x_0)) \leq C
\begin{cases}
r^{n-sp} &\text{if } p < n/s, \\
(\log(R/r))^{1-p} &\text{if } p = n/s.
\end{cases}
\end{split}
\end{equation}
By combining \eqref{lb}, \eqref{eq:vr-tail}, \eqref{eq:poincare} and \eqref{eq:capacity}, we arrive at
\begin{equation*}
\wsup_{B_{\rho/2}(z)}v_r \leq 
\begin{cases}
\displaystyle C_1 \delta + C_2 \delta^{-\frac{(p-1)n}{sp^2}} \rho^{-\frac{n}{p}} R^{s} {r}^{\frac{n-sp}{p}} &\text{if } p < n/s, \\
\displaystyle C_1 \delta + C_2 \delta^{-\frac{(p-1)n}{sp^2}}\rho^{-\frac{n}{p}} R^{s} \left( \log \frac{R}{r} \right)^{\frac{1-p}{p}} &\text{if } p = n/s.
\end{cases}
\end{equation*}
We first take $\delta > 0$ sufficiently small so that $C_1 \delta \leq 1/4$. By taking $r_0 = r_0(n, s, p, \rho, \Lambda, R)>0$ sufficiently small, we conclude that $\wsup_{B_{\rho/2}(z)}v_r \leq 1/2$ for all $r < r_0$, from which \eqref{eq:claim2} follows.
\end{proof}

With the help of \Cref{irr}, we finally prove the necessary part of \Cref{thm:main}.

\begin{proof} [Proof of the necessary part of \Cref{thm:main}]
Let $x_0 \in \partial \Omega$ be a boundary point. For simplicity, we will write $B_\rho=B_\rho(x_0)$ for $\rho>0$. We claim that $x_0 \in \partial \Omega$ is irregular if the Wiener integral converges, that is,
\begin{align} \label{wieint}
\int_0 \left(\frac{\mathrm{cap}_{s,p}(D_t, B_{2t})}{t^{n-sp}}\right)^{\frac{1}{p-1}} \,\frac{\mathrm{d}t}{t} < +\infty,
\end{align}
where $D_t$ is defined as \eqref{eq:D}. To prove this, we shall employ \Cref{irr}; i.e., we will show that there is a $\rho \in (0,1)$ such that the condition \eqref{irreq} holds for the $\mathcal{L}$-potential $u_{\rho}(x)=\mathfrak{R}(D_{\rho}, B_{8\rho})$ of $D_{\rho}$ in $B_{8\rho}$.

Let $\mu_\rho$ be the $\mathcal{L}$-distribution of $D_\rho$ with respect to $B_{8\rho}$, that is, $\mathcal{L}u_\rho=\mu_\rho$ in $B_{8\rho}$. By \Cref{thm:regularization} and \Cref{rmk:supersoln} (ii), we may assume that $u$ is $\mathcal{L}$-superharmonic in $B_{8\rho}$. Then, \Cref{thm:Wolff} yields
\begin{equation} \label{pt}
u_{\rho}(x_0) \leq C \left( \inf_{B_{2\rho}}u_{\rho} + {\bf W}_{s,p}^{\mu_{\rho}}(x_0, 4\rho) + \mathrm{Tail}(u_{\rho}; x_0, 2\rho) \right).
\end{equation}
By \Cref{mecap}, we have
\begin{equation} \label{eq:W}
\begin{split}
{\bf W}^{\mu_{\rho}}_{s, p}(x_0, 4\rho)
&\leq \int_0^{4\rho} \left(\frac{\mu_\rho(\overline{B_{t}})}{t^{n-sp}}\right)^{\frac{1}{p-1}} \frac{\mathrm{d}t}{t} \\
&\leq \Lambda^{\frac{1}{p-1}} \int_0^{4\rho} \left(\frac{\mathrm{cap}_{s, p}(D_t, B_{8\rho})}{t^{n-sp}}\right)^{\frac{1}{p-1}} \frac{\mathrm{d}t}{t} \\
&\leq \Lambda^{\frac{1}{p-1}} \int_0^{4\rho}\left(\frac{\mathrm{cap}_{s,p}(D_t, B_{2t})}{t^{n-sp}}\right)^{\frac{1}{p-1}} \,\frac{\mathrm{d}t}{t}.
\end{split}
\end{equation}
Thus, it follows from \eqref{wieint} that ${\bf W}_{s,p}^{\mu_{\rho}}(x_0, 4\rho) \to 0$ as $\rho \to 0$.

For the first term on the right-hand side of \eqref{pt}, we claim that
\begin{equation} \label{eq:inf-mu}
\inf_{B_{2\rho}}u_{\rho}  \leq C \left( \frac{\mathrm{cap}_{s, p}(D_\rho, B_{2\rho})}{\rho^{n-sp}} \right)^{\frac{1}{p-1}}.
\end{equation}
Let $\lambda_\rho = \essinf_{B_{2\rho}} u_\rho$ and define $v_\rho(x) = \min\lbrace u_\rho(x), \lambda_\rho \rbrace$. By testing the equation $\mathcal{L}u_{\rho}=\mu_{\rho}$ with $v_{\rho}$, we obtain
\begin{align} \label{comp1}
\mathcal{E}^k(u_{\rho}, v_{\rho})=\int_{B_{8\rho}} v_{\rho} \,\mathrm{d}\mu_\rho \leq \lambda_{\rho} \mu_{\rho}(\overline{B_{\rho}})
\end{align}
since the support of $\mu_{\rho}$ is contained in $\overline{B_{\rho}}$. Moreover, we claim that 
\begin{align*}
	|u_{\rho}(x)-u_{\rho}(y)|^{p-2}(u_{\rho}(x)-u_{\rho}(y)) (v_{\rho}(x)-v_{\rho}(y)) 
	\geq |v_{\rho}(x)-v_{\rho}(y)|^p.
\end{align*}
Indeed, we may assume $u_{\rho}(x) \geq u_{\rho}(y)$ without loss of generality. Then,
\begin{align*}
	&|u_{\rho}(x)-u_{\rho}(y)|^{p-2}(u_{\rho}(x)-u_{\rho}(y)) (v_{\rho}(x)-v_{\rho}(y))\\
	&=	
	\left\{ \begin{array}{ll} 
		0, & \textrm{if $u_{\rho}(x) \geq u_{\rho}(y) \geq \lambda_{\rho}$},\\
		(u_{\rho}(x)-u_{\rho}(y))^{p-1} (\lambda_{\rho}-v_{\rho}(y)), & \textrm{if $u_{\rho}(x) >\lambda_{\rho} \geq u_{\rho}(y)$},\\
		|v_{\rho}(x)-v_{\rho}(y)|^p, & \textrm{if $\lambda_{\rho} \geq u_{\rho}(x) \geq u_{\rho}(y)$}
	\end{array} \right.\\
	&\geq |v_{\rho}(x)-v_{\rho}(y)|^p.
\end{align*}
Therefore, we have
\begin{equation} \label{eq:vv}
\mathcal{E}^k(v_{\rho}, v_{\rho}) \leq \mathcal{E}^k(u_{\rho}, v_{\rho}).
\end{equation}
We may assume without loss of generality that $\lambda_\rho > 0$. Then, it is easily checked that $v_{\rho} \in W_0^{s,p}(B_{8\rho})$ and $v_{\rho} \equiv \lambda_{\rho}$ on $\overline{B_{2\rho}}$. Thus, $v_{\rho}/\lambda_{\rho}$ is admissible for $(s, p)$-capacity of $\overline{B_{\rho}}$ with respect to $B_{8\rho}$. Therefore, the ellipticity condition \eqref{eq:ellipticity} yields
\begin{align} \label{comp2}
\mathrm{cap}_{s,p}(\overline{B_{\rho}}, B_{8\rho}) \leq \mathcal{E}^{s,p}(v_{\rho}/\lambda_{\rho}, v_{\rho}/\lambda_{\rho}) =\lambda_{\rho}^{-p} \mathcal{E}^{s,p}(v_{\rho}, v_{\rho}) \leq \frac{\Lambda}{\lambda_{\rho}^p} \mathcal{E}^k(v_\rho, v_\rho).
\end{align}
By combining \eqref{comp1}, \eqref{eq:vv}, \eqref{comp2}, and \Cref{lem:cap}, we obtain
\begin{align*}
\rho^{n-sp} \leq C \lambda_{\rho}^{-(p-1)} \mu_{\rho}(\overline{B_{\rho}}).
\end{align*}
This inequality, together with \Cref{mecap}, proves \eqref{eq:inf-mu}.

We next show that 
\begin{align}\label{liminf}
\liminf_{\rho \to 0}\frac{\mathrm{cap}_{s,p}(D_\rho, B_{2\rho})}{\rho^{n-sp}}=0.
\end{align}
Assume to the contrary that there exists a constant $\alpha > 0$ such that
\begin{equation*}
\liminf_{\rho \to 0}\frac{\mathrm{cap}_{s,p}(D_\rho, B_{2\rho})}{\rho^{n-sp}} \geq 2\alpha.
\end{equation*}
Then, there exists a constant $\rho_0>0$ such that 
\begin{align*}
\frac{\mathrm{cap}_{s,p}(D_{\rho}, B_{2\rho})}{\rho^{n-sp}} \geq \alpha \quad \textrm{for all $\rho \in (0, \rho_0)$.}
\end{align*}
Thus, we have
\begin{align*}
\left( \frac{\mathrm{cap}_{s,p}(D_\rho, B_{2\rho})}{\rho^{n-sp}} \right)^{1/(p-1)} \frac{1}{\rho} \geq \frac{\alpha^{1/(p-1)}}{\rho} \quad\text{for all } \rho \in (0, \rho_0),
\end{align*}
which contradicts to \eqref{wieint}.

It only remains to estimate the tail term in \eqref{pt}. Since $u_\rho \in W^{s, p}_0(B_{8\rho})$, we have
\begin{equation*}
\mathrm{Tail}(u_{\rho}; x_0, 2\rho) = \left( (2\rho)^{sp} \int_{B_{8\rho} \setminus B_{2\rho}} \frac{u_\rho^{p-1}(y)}{|y-x_0|^{n+sp}} \,\mathrm{d}y \right)^{\frac{1}{p-1}}.
\end{equation*}
We observe that, by minimizing property of the $\mathcal{L}$-potential $u_\rho$ (\Cref{lem:potential} (iii)) and the ellipticity assumption \eqref{eq:ellipticity}, we have
\begin{equation*}
\mathcal{E}^k(u_\rho) \leq \mathcal{E}^k(v) \leq \Lambda \mathcal{E}^{s, p}(v)
\end{equation*}
for all functions $v \in W_0(D_\rho, B_{8\rho})$ admissible for $\mathrm{cap}_{s, p}(D_\rho, B_{8\rho})$. Taking the infimum over $v \in W_0(D_\rho, B_{8\rho})$, we obatin
\begin{equation*}
\Lambda \, \mathrm{cap}_{s, p}(D_\rho, B_{8\rho}) \geq \int_{B_{10\rho} \setminus B_{8\rho}} \int_{B_{8\rho} \setminus B_{2\rho}} |u_{\rho}(x)-u_{\rho}(y)|^p k(x, y) \,\mathrm{d}y \,\mathrm{d}x.
\end{equation*}
For $x \in B_{10\rho} \setminus B_{8\rho}$ and $y \in B_{8\rho} \setminus B_{2\rho}$, it holds that
\begin{equation*}
|x-y| \leq 18\rho \leq 9|y-x_0|.
\end{equation*}
Thus, we obtain
\begin{equation*}
\begin{split}
\Lambda \mathrm{cap}_{s, p}(D_\rho, B_{8\rho}) 
&\geq \frac{1}{\Lambda 9^{n+sp}} \int_{B_{10\rho} \setminus B_{8\rho}} \int_{B_{8\rho} \setminus B_{2\rho}} \frac{u_{\rho}^p(y)}{|y-x_0|^{n+sp}} \,\mathrm{d}y \,\mathrm{d}x \\
&\geq C \rho^{n} \int_{B_{8\rho} \setminus B_{2\rho}} \frac{u_{\rho}^p(y)}{|y-x_0|^{n+sp}} \,\mathrm{d}y.
\end{split}
\end{equation*}
By applying Young's inequality, for any $\varepsilon > 0$ we have
\begin{equation*}
\begin{split}
\int_{B_{8\rho} \setminus B_{2\rho}} \frac{u_{\rho}^{p-1}(y)}{|y-x_0|^{n+sp}} \,\mathrm{d}y
&\leq \frac{p-1}{p} \varepsilon^{-\frac{p}{p-1}} \int_{B_{8\rho} \setminus B_{2\rho}} \frac{u_{\rho}^p(y)}{|y-x_0|^{n+sp}} \,\mathrm{d}y + \frac{1}{p} \varepsilon^p \int_{B_{8\rho} \setminus B_{2\rho}} \frac{\mathrm{d}y}{|y-x_0|^{n+sp}} \\
&\leq C \varepsilon^{-\frac{p}{p-1}} \rho^{-n} \mathrm{cap}_{s, p}(D_\rho, B_{2\rho}) + C \varepsilon^p \rho^{-sp}.
\end{split}
\end{equation*}
Therefore, we estimate
\begin{equation*}
\mathrm{Tail}(u_\rho; x_0, 2\rho) \leq C \left( \varepsilon^{-\frac{p}{p-1}} \frac{\mathrm{cap}_{s, p}(D_\rho, B_{2\rho})}{\rho^{n-sp}} + \varepsilon^p \right)^{\frac{1}{p-1}}.
\end{equation*}
Taking $\varepsilon$ sufficiently small and then using \eqref{liminf}, we can make $\mathrm{Tail}(u_\rho; x_0, 2\rho)$ as small as we want.

There exists a small $\rho > 0$ such that
\begin{align*}
u(x_0)=u_{\rho}(x_0)<1,
\end{align*}
as required. We conclude that $x_0$ is an irregular boundary point by \Cref{irr}.
\end{proof}

\begin{appendix}


\section{Algebraic inequalities} \label{sec:inequalities}


In this section, we provide some algebraic inequalities that are used in the Caccioppoli-type estimates. Let $p \in (1, \infty)$ and $\beta, \gamma \in \mathbb{R}$ be such that $\beta = \gamma-(p-1)$. Recall that, in the local case, we use the following inequalities for the proof of Caccioppoli-type estimates \cite[Theorem 8.25 and 8.26]{GT01}: assume $\beta \neq 0$, $\gamma \neq 0$, and let $l \geq 0$. Then, there exists constants $c, C > 0$, depending only on $p$, such that
\begin{equation*}
\begin{split}
|\nabla u|^{p-2} \nabla u \cdot \nabla (f(u) \eta^p)
&\geq |\nabla F(u)|^p \eta^p - \frac{|\gamma|}{|\beta|} |\nabla F(u)|^{p-1} \eta^{p-1} |F(u)| |\nabla \eta| \\
&\geq \frac{1}{p} \left( |\nabla F(u)|^p \eta^p - \left( \frac{|\gamma|}{|\beta|} \right)^p |F(u)|^{p} |\nabla \eta|^p \right) \\
&\geq c \, |\nabla (F(u) \eta)|^p - C \left( 1+\left( \frac{|\gamma|}{|\beta|} \right)^p \right) |F(u)|^p |\nabla \eta|^p,
\end{split}
\end{equation*}
where $f, F: I \to [0, \infty)$ are defined by
\begin{equation} \label{eq:f}
f(t) = \frac{1}{\beta} \left( (t+d)^{\beta} - (l+d)^{\beta} \right) \quad\text{and}\quad F(t) = \frac{p}{\gamma} (t+d)^{\gamma/p}.
\end{equation}
Here, $I = [l, \infty)$ when $\beta > 0$ and $I = [0, l]$ when $\beta < 0$. As a discrete version, we prove the following algebraic inequality.

\begin{lemma} \label{lem:alg-ineq1}
Assume that $\beta \neq 0$ and $\gamma \neq 0$. There exist $c, C > 0$, depending only on $p$, such that
\begin{equation*}
\begin{split}
&|a-b|^{p-2}(a-b) (f(a) \eta_1^p - f(b) \eta_2^p) \\
&\geq c \left| F(a) \eta_1 - F(b) \eta_2 \right|^p - C \left( 1+ \left( \frac{|\gamma|}{|\beta|} \right)^p \right) \max \lbrace |F(a)|, |F(b)| \rbrace^p |\eta_1-\eta_2|^p
\end{split}
\end{equation*}
for any $a, b \in I$ and $\eta_1, \eta_2 \geq 0$.
\end{lemma}

Some inequalities similar to \Cref{lem:alg-ineq1} are known for almost all values of $\beta$. See, for instance, \cite{BP16} for $\beta \geq 1$ and \cite{CK22} for $\beta < -(p-1)$. \Cref{lem:alg-ineq1} can be proved in a similar way, but let us provide a proof to make the paper self-contained. To this end, we first prove the following two lemmas as intermediate steps.

\begin{lemma} \label{lem:alg-ineq2}
If $\beta \neq 0$ and $\gamma \neq 0$, then
\begin{equation*}
|a-b|^{p-2}(a-b)(f(a)-f(b)) \geq |F(a)-F(b)|^p
\end{equation*}
and
\begin{equation*}
|a-b|^{p-1} \min \left\lbrace F'(a), F'(b) \right\rbrace^{p-1} \leq |F(a)-F(b)|^{p-1}
\end{equation*}
for any $a, b \in I$.
\end{lemma}

\begin{proof}
We may assume that $a>b$. Then, we have
\begin{equation*}
\left| \frac{F(a) - F(b)}{a-b} \right|^p = \left| \fint_{b}^{a} F'(t) \,\mathrm{d}t \right|^p \leq \fint_{b}^{a} f'(t) \,\mathrm{d}t = \frac{f(a) - f(b)}{a-b}
\end{equation*}
by Jensen's inequality, and
\begin{equation*}
\left| \frac{F(a) - F(b)}{a-b} \right|^{p-1} = \left| \fint_{b}^{a} F'(t) \,\mathrm{d}t \right|^{p-1} \geq \min \lbrace F'(a), F'(b) \rbrace^{p-1}
\end{equation*}
since $F'$ is positive on $[0, \infty)$.
\end{proof}

\begin{lemma} \label{lem:alg-ineq3}
Let $A, B \in \mathbb{R}$ and $\eta_1, \eta_2 \geq 0$, then
\begin{equation*}
\begin{split}
|A-B|^p \min\lbrace \eta_1, \eta_2 \rbrace^p
&\geq 2^{1-p} |A \eta_1 - B \eta_2|^p - \max\lbrace |A|, |B| \rbrace^p |\eta_1-\eta_2|^p \quad\text{and} \\
|A-B|^p \max\lbrace \eta_1, \eta_2 \rbrace^p
&\leq 2^{p-1} |A \eta_1 - B \eta_2|^p + 2^{p-1} \max\lbrace |A|, |B| \rbrace^p |\eta_1-\eta_2|^p.
\end{split}
\end{equation*}
\end{lemma}

\begin{proof}
We may assume that $\eta_1 \geq \eta_2$. Then, the desired inequalities follow from the equalities
\begin{equation*}
A\eta_1 - B\eta_2 = (A-B) \eta_2 + A (\eta_1-\eta_2) = (A-B) \eta_1 + B (\eta_1-\eta_2)
\end{equation*}
and the triangle inequality.
\end{proof}

We prove \Cref{lem:alg-ineq1} by using \Cref{lem:alg-ineq2} and \Cref{lem:alg-ineq3}.

\begin{proof} [Proof of \Cref{lem:alg-ineq1}]
We may assume that $a > b$. Then, we have
\begin{align}
(a-b)^{p-1} (f(a) \eta_1^p - f(b) \eta_2^p)
&= (a-b)^{p-1}(f(a) - f(b))\eta_1^p + (a-b)^{p-1} f(b)(\eta_1^p-\eta_2^p) \label{eq:min} \\ 
&= (a-b)^{p-1}(f(a) - f(b))\eta_2^p + (a-b)^{p-1} f(a)(\eta_1^p-\eta_2^p).\label{eq:max}
\end{align}
We apply \Cref{lem:alg-ineq2} to \eqref{eq:min} if $\beta \geq 1$ or to \eqref{eq:max} if $\beta \leq 1$. Then, we obtain
\begin{equation*}
\begin{split}
J
:=&~ (a-b)^{p-1} (f(a) \eta_1^p - f(b) \eta_2^p) \\
\geq&~ |F(a)-F(b)|^p \min\lbrace \eta_1, \eta_2 \rbrace^p - |F(a)-F(b)|^{p-1} \max\left\lbrace \frac{|f(a)|}{F'(a)^{p-1}}, \frac{|f(b)|}{F'(b)^{p-1}} \right\rbrace |\eta_1^p-\eta_2^p|.
\end{split}
\end{equation*}
By using
\begin{equation*}
\frac{|f(t)|}{F'(t)^{p-1}} \leq \frac{|\gamma|}{p|\beta|} |F(t)|, \quad |\eta_1^p-\eta_2^p| \leq p |\eta_1-\eta_2|\max\lbrace \eta_1, \eta_2 \rbrace^{p-1},
\end{equation*}
and Young's inequality, we deduce
\begin{equation*}
\begin{split}
J
&\geq |F(a)-F(b)|^p \min\lbrace \eta_1, \eta_2 \rbrace^p \\
&\quad - \frac{1}{2^{2p-1}} |F(a)-F(b)|^p \max\lbrace\eta_1, \eta_2 \rbrace^p - C \left( \frac{|\gamma|}{|\beta|} \right)^p \max\lbrace |F(a)|, |F(b)| \rbrace^p |\eta_1-\eta_2|^p,
\end{split}
\end{equation*}
where $C = C(p) > 0$. Applying \Cref{lem:alg-ineq3} with $A = F(a)$ and $B = F(b)$ finishes the proof.
\end{proof}

The case $\gamma = 0$, or equivalently $\beta = -(p-1)$, is treated in the following lemma, which is a discrete version of
\begin{equation*}
|\nabla u|^{p-2} \nabla u \cdot \nabla (f(u) \eta^p) \geq c |\nabla \log (u+d)|^p \eta^p - C |\nabla \eta|^p,
\end{equation*}
where $f$ is given by \eqref{eq:f} with $\beta = -(p-1)$ and $c, C > 0$ are constants depending only on $p$.

\begin{lemma} \label{lem:alg-ineq-log}
Assume that $\gamma = \beta + p-1 = 0$. There exist $c, C > 0$, depending only on $p$, such that
\begin{equation*}
|a-b|^{p-2}(a-b) (f(a)\eta_1^p - f(b)\eta_2^p) \geq c\, |\log (a+d) - \log(b+d)|^p \min\lbrace \eta_1, \eta_2 \rbrace^p - C |\eta_1-\eta_2|^p
\end{equation*}
for any $a, b \in I$ and $\eta_1, \eta_2 \geq 0$.
\end{lemma}

\begin{proof}
We may assume that $a > b$. By applying \cite[Lemma 3.1]{DCKP16} to $\eta_1$ and $\eta_2$ with
\begin{equation*}
\varepsilon = \delta \frac{a-b}{a+d} \in (0,1), \quad \delta \in (0, 1),
\end{equation*}
we have
\begin{equation*}
\eta_1^p - \eta_2^p \leq C \varepsilon \eta_2^p + C \varepsilon^{1-p} |\eta_1-\eta_2|^p,
\end{equation*}
where $C = C(p) > 0$. Thus, it follows from \eqref{eq:max} and $f(a) \geq \frac{1}{1-p}(a+d)^{1-p}$ that
\begin{equation*}
\begin{split}
K
:=&~ (a-b)^{p-1}(f(a)\eta_1^p - f(b)\eta_2^p) \\
\geq&~ (a-b)^{p-1} (f(a)-f(b)) \eta_2^p - C\delta \left( \frac{a-b}{a+d} \right)^p \eta_2^p - C \delta^{1-p} |\eta_1-\eta_2|^p.
\end{split}
\end{equation*}
By using the same arguments as in the proof of \Cref{lem:alg-ineq2}, we obtain
\begin{equation*}
\left( \frac{a-b}{a+d} \right)^{p} \leq |\log (a+d) - \log (b+d)|^p \leq (a-b)^{p-1}(f(a)-f(b)).
\end{equation*}
Therefore, we deduce that
\begin{equation*}
K \geq (1-C\delta) |\log (a+d) - \log (b+d)|^p \eta_2^p - C \delta^{1-p} |\eta_1-\eta_2|^p,
\end{equation*}
from which we conclude the lemma by taking $\delta$ sufficiently small.
\end{proof}

We also provide the following algebraic inequality, which is similar to \Cref{lem:alg-ineq1}.

\begin{lemma} \label{lem:alg-ineq-wolff}
Let $a, b \in \mathbb{R}$, $\eta_1, \eta_2 \geq 0$, and assume that $\gamma$ satisfies \eqref{eq:gamma}. Let $\tau = \gamma/(p-1)$ and $q = p\gamma/(p-\gamma/(p-1))$, then
\begin{equation*}
\begin{split}
&|a-b|^{p-2}(a-b)(g(a_+) \eta_1^p - g(b_+) \eta_2^p) \\
&\geq c\, |G(a_+) \eta_1 - G(b_+)\eta_2|^p - C \max\lbrace (1+a_+)^{\gamma}, (1+b_+)^{\gamma} \rbrace {\bf 1}_{\lbrace a>0 \rbrace} {\bf 1}_{\lbrace b>0 \rbrace} |\eta_1 - \eta_2|^p,
\end{split}
\end{equation*}
where
\begin{equation*}
g(t) = \frac{1}{\tau-1}(1-(1+t)^{1-\tau}), \quad G(t) = \frac{q}{\gamma}((1+t)^{\gamma/q}-1),
\end{equation*}
and $c, C > 0$ are constants depending only on $p$.
\end{lemma}

\begin{proof}
We may assume that $a>b$. The same arguments as in the proof of \Cref{lem:alg-ineq2} show that
\begin{equation} \label{eq:alg-ineq2}
(a_+-b_+)^{p-1} (g(a_+) - g(b_+)) \geq |G(a_+) - G(b_+)|^p
\end{equation}
and
\begin{equation*}
(a_+-b_+)^{p-1} g'(a_+)^{\frac{p-1}{p}} \leq |G(a_+)-G(b_+)|^{p-1}.
\end{equation*}
The case $b \leq 0$ follows from \eqref{eq:alg-ineq2}. Thus, we assume $b > 0$. Since $a-b \geq a_+ - b_+$, $(a-b)^{p-1} g(b_+) = (a_+ - b_+)^{p-1} g(b_+)$, and $g(b_+) \leq 1$, we obtain
\begin{equation*}
\begin{split}
L
:=&~ (a-b)^{p-1} (g(a_+) \eta_1^p - g(b_+) \eta_2^p) \\
=&~ (a-b)^{p-1} ((g(a_+) - g(b_+)) \eta_1^p + g(b_+) (\eta_1^p - \eta_2^p)) \\
\geq&~ (a_+-b_+)^{p-1} ((g(a_+) - g(b_+)) \eta_1^p + g(b_+) (\eta_1^p - \eta_2^p)) \\
\geq&~ |G(a_+) - G(b_+)|^p \eta_1^p - p(a_+-b_+)^{p-1} |\eta_1-\eta_2| \max\lbrace \eta_1, \eta_2 \rbrace^{p-1} \\
\geq&~ |G(a_+) - G(b_+)|^p \eta_1^p - p |G(a_+)-G(b_+)|^{p-1} \max\lbrace \eta_1, \eta_2 \rbrace^{p-1} g'(a_+)^{-\frac{p-1}{p}} |\eta_1-\eta_2|.
\end{split}
\end{equation*}
By using Young's inequality, \Cref{lem:alg-ineq3}, and $G(t)^p \leq C (1+t)^{\gamma}$, we conclude that
\begin{equation*}
\begin{split}
L
&\geq |G(a_+) - G(b_+)|^p \min\lbrace \eta_1, \eta_2 \rbrace^p \\
&\quad - \frac{1}{2^{2p-1}} |G(a_+)-G(b_+)|^{p} \max\lbrace \eta_1, \eta_2 \rbrace^{p} - C \frac{1}{g'(a_+)^{p-1}} |\eta_1-\eta_2|^p \\
&\geq \frac{1}{2^{p}} |G(a_+)\eta_1 - G(b_+)\eta_2|^p - C \max\lbrace (1+a_+)^{\gamma}, (1+b_+)^{\gamma} \rbrace |\eta_1 - \eta_2|^p,
\end{split}
\end{equation*}
where $C = C(p) > 0$.
\end{proof}

Let us finish the section with one more algebraic inequality.

\begin{lemma} \label{lem:alg-ineq-suff}
Let $a, b > 0$ and assume that $\gamma < p$, $\gamma \neq 0$. Then,
\begin{equation*}
\max\lbrace a, b \rbrace^{\gamma-p} |a-b|^p \leq \left( \frac{p}{|\gamma|} \right)^p |a^{\gamma/p} - b^{\gamma/p}|^p.
\end{equation*}
\end{lemma}

\begin{proof}
We may assume without loss of generality that $a > b$. Then,
\begin{equation*}
|a^{\gamma/p}-b^{\gamma/p}|^p = \left| \int_b^a \frac{\gamma}{p} t^{\gamma/p-1} \,\mathrm{d}t \right|^p \geq \left|\frac{\gamma}{p} \right|^p a^{\gamma-p} (a-b)^p,
\end{equation*}
where we use the monotonicity of $t \mapsto t^{\gamma/p-1}$.
\end{proof}

\end{appendix}


\end{document}